\documentclass[12pt]{article}

\usepackage[utf8]{inputenc}
\usepackage{tikz}
\usepackage[T1]{fontenc}
\usepackage{dsfont}
\usepackage[a4paper, margin=2.7cm]{geometry}
\usepackage{amsmath, amsthm, amsfonts, amssymb}
\usepackage{mathrsfs}        
\usepackage{enumerate}
\usepackage{esint}
\usepackage{caption}
\usepackage{dsfont}
\usepackage{xcolor}
\usepackage[colorlinks=true,linkcolor=blue,citecolor=blue,urlcolor=blue,breaklinks]{hyperref}
\usepackage{bbm} 
\usepackage{thmtools}
\usepackage{natbib}
\usepackage{mathtools}
\usepackage{url}

\newcommand{\abs}[1]{\left\vert#1\right\vert}           
\newcommand{\ap}[1]{\left\langle#1\right\rangle}        
\newcommand{\fl}[1]{\left\lfloor#1\right\rfloor}        
\newcommand{\cl}[1]{\left\lceil#1\right\rceil}          
\newcommand{\norm}[1]{\left\Vert#1\right\Vert}          
\newcommand{\tnorm}[1]{\left\vert\!\left\vert\!\left\vert#1\right\vert\!\right\vert\!\right\vert}

\newcommand\myeq{\stackrel{\mathclap{\normalfont\mbox{*}}}{=}}

\renewcommand{\div}{\operatorname{div}}       

\DeclareMathOperator{\Lip}{Lip}
\DeclareMathOperator{\dist}{dist}

\def\d{\,\mathrm{d}}
\def \ddt{\frac{\mathrm{d}}{\mathrm{d}t}}
\def\ird{\int_{\R^d}}

\def\R{\mathbb{R}}
\def\M{\mathcal{M}}

\def\Var{{\textrm{Var}}\,}

\def\Id{{\textrm{Id}}\,}

\def\e{\varepsilon}
\def\f{\varphi}

\def\:{\colon}  

\def\Jn{J_\varepsilon^{t_1,\dots,t_n}}
\def\JJi{\hat{f}\Big(\frac{\xi\sigma_j}{\sqrt n\bar\sigma}\Big)}
\def\GGi{\hat{G}\Big(\frac{\xi\sigma_j}{\sqrt n\bar\sigma}\Big)}
\def\JJj{\hat{f}\Big(\frac{\xi\sigma_k}{\sqrt n\bar\sigma}\Big)}
\def\GGj{\hat{G}\Big(\frac{\xi\sigma_k}{\sqrt n\bar\sigma}\Big)}
\def\Jnorm{\mathfrak{f}_n}
\def\Jnormh{\hat{\mathfrak{f}}_n}

\newtheorem{theorem}{Theorem}[section]
\newtheorem{corollary}[theorem]{Corollary}

\newtheorem{lemma}[theorem]{Lemma}
\newtheorem{proposition}[theorem]{Proposition}

\theoremstyle{definition}
\newtheorem{definition}[theorem]{Definition}
\newtheorem{remark}[theorem]{Remark}

\newtheorem{hyp}{Hypothesis}

\def\thetitle{A uniform-in-time nonlocal approximation of the
	standard Fokker-Planck equation}
\def\theauthor{José A.~Cañizo \& Niccolò Tassi}

\hypersetup{pdftitle={\thetitle}}
\hypersetup{pdfauthor={\theauthor}}

\title{\thetitle}
\author{\theauthor}
\date{}

\AtEndDocument{\vspace{2cm}{\footnotesize
		\noindent
		\textsc{José A. Cañizo. Departamento de Matemática Aplicada \& IMAG,
			Universidad de Granada. Avda. Fuentenueva S/N, 18071 Granada, Spain.}
		\textit{E-mail address}: \texttt{\href{mailto:canizo@ugr.es}{canizo@ugr.es}}
		\par
		\addvspace{\medskipamount}
		\noindent
		\textsc{Niccolò Tassi. Departamento de Matemática Aplicada \& IMAG,
			Universidad de Granada. Avda. Fuentenueva S/N, 18071 Granada, Spain.}
		\textit{E-mail address}: \texttt{\href{mailto:tassi@ugr.es}{tassi@ugr.es}}
}}

\begin{document}
\maketitle

\begin{abstract}
  We study a nonlocal approximation of the Fokker-Planck equation in
  which we can estimate the speed of convergence to equilibrium in a
  way which does not degenerate as we approach the local limit of the
  equation. This uniform estimate cannot be easily obtained with
  standard inequalities or entropy methods, but can be obtained
  through the use of Harris’s theorem, finding interesting links to
  quantitative versions of the central limit theorem in
  probability. As a consequence one can prove that solutions of this
  nonlocal approximation converge to solutions of the usual
  Fokker-Planck equation uniformly in time---hence we show the
  approximation is asymptotic-preserving in this sense. The associated
  equilibrium has some interesting tail and regularity properties,
  which we also study.
\end{abstract}
\setcounter{tocdepth}{2}
\tableofcontents

\section{Introduction}
We consider the linear partial differential equation
\begin{equation}
  \label{NLFP} \tag{NLFP}
  \partial_t u
  = \frac{1}{\e^2} \big( J_\e*u-u \big) + \div(xu)  
    =: L_\e u,
\end{equation}
where $u \: [0,+\infty) \times \R^d \to \R$ is the unknown whose
variables we usually name by $u \equiv u(t,x)$, and $\e > 0$ is a
scaling parameter. This PDE involves a nonnegative kernel
$J \: \R^d \to [0,+\infty)$ which we always assume at least
integrable,  
and such that
\begin{equation}
  \label{eq:J-moments-basic}
  \begin{split}    
    \ird J(x) \d x &= 1,
    \qquad \ird J(x) x \d x=0
    \qquad \ird J(x) x_i x_j \d x = 2 \delta_{ij},
  \end{split}
\end{equation}
for $i,j = 1, \dots, d$. Making straightforward scaling arguments one
can also apply our results if $J$ is not of mass $1$.

A canonical example of $J$ for which our results are meaningful is any
smooth, compactly supported probability distribution satisfying the
above, which is a helpful case to keep in mind. The scaled kernel
$J_\e$ is defined by
\begin{equation*}
  J_\e(x) := \e^{-d} J\Big( \frac{x}{\e} \Big),
  \qquad x \in \R^d,
\end{equation*}
and the convolution in (\ref{NLFP}) is understood to be in the $x$
variable. Thus \eqref{NLFP} is, strictly speaking, a partial
integro-differential equation since it involves an integral kernel. We
choose this particular scaling in $\e$ since it is well known that as
$\e \to 0$ the PDE (\ref{NLFP}) resembles the standard Fokker-Planck
PDE
\begin{equation}
  \label{eq:FP}\tag{FP}
  \partial_t u = \Delta u + \div(xu)=:L_0u.
\end{equation}
We refer to the
book by \cite{AndreuVaillo2010} for a good explanation of this limit.
Hence, equation (\ref{NLFP}) can be understood as a nonlocal version of
the standard Fokker-Planck equation.

Our interest in studying equation \eqref{NLFP} is not directly due to
the model itself, but as a proof of concept of a new method to study
nonlocal-to-local limits. Equation \eqref{NLFP} is a nonlocal
diffusion which is interesting enough to have non-explicit equilibria,
which likely do not have Gaussian decay (see Section
\ref{sec:tails}). Standard entropy methods are hard to apply to this
PDE, especially in the limit $\e \to 0$. Though the equation can be
studied by Fourier methods (see \cite{AndreuVaillo2010}), we are
interested in methods which do not depend on an explicit Fourier
transform solution. This said, we also point out that the drift
$\div (x u)$ appears in a natural way if we consider the PDE
$\partial_t v = J*v - v$ and then carry out the standard diffusive change of
variables $u(t,x) = e^{d t} v(\frac12 e^{2t}, x e^t)$. The equation
for $u$ then becomes
\begin{equation}
    \partial_t u = \frac{1}{\e(t)^2} \left( J_{\e(t)} * u - u \right) + \div(xu),
\end{equation}
where now $\e(t) = e^{-t}$. Equation \eqref{NLFP} is then the result of ``freezing'' the parameter $\e$ in this equation, and as such is a natural model to understand regarding nonlocal diffusion.

In contrast with local diffusion, the nonlocal PDE (\ref{NLFP})
does not have instantaneous regularization properties. Both PDE (\ref{NLFP}) and
(\ref{eq:FP}) have a stochastic interpretation in terms of a Markov
process; we refer to standard texts such as \cite{Risken1996} for
this. In this paper we are interested in understanding the asymptotic
behavior of solutions to \eqref{NLFP} as $t \to +\infty$ and as $\e \to 0$ by the use of
a new method based on Harris's theorem, since techniques for nonlocal
PDE still have many shortcomings and this equation is a good test
case. We pay particular attention to whether the asymptotic behavior
of \eqref{NLFP} matches that of the standard Fokker-Planck equation as
$\e \to 0$. The use of Harris's theorem is by no means new in the
study of nonlocal PDE, but as far as we know it is the first time that
it can be applied to give a result which is stable under the scaling
in (\ref{NLFP}). We consider this the main novelty of this paper.

Obviously, since our results will be valid also for a fixed $\e$, we
may frame the problem in a slightly different way by studying the PDE
\begin{equation*}
  \partial_t u
   = c \big( J*u-u \big) + \div(xu),
\end{equation*}
    for $J$ a general probability distribution satisfying \eqref{eq:J-moments-basic}. We have chosen to write the PDE as in
\eqref{NLFP} in order to emphasize that we are especially interested
in its scaling limit to the Fokker-Planck equation.

\paragraph{Summary of main results}

Here are the two main results we will prove in this paper:
\begin{enumerate}
\item Solutions to (\ref{NLFP}) converge to equilibrium exponentially:
  \begin{equation*}
    \| u(t, \cdot) - F_\e \| \leq C e^{-\gamma t} \|u_0 -
    F_\e\|,
    \qquad t \geq 0
  \end{equation*}
  in a variety of weighted $L^1$ norms $\|\cdot\|$, where $F_\e$
  is a suitable equilibrium of the PDE. Our main point is that \emph{both
  constants $C$ and $\gamma$ can be taken to be independent of
  $\e$} as $\e \to 0$. This is not contained in previous
  results, except for \cite{mischler_uniform_2017}, who use different
  techniques.
  
\item If $u$ is the solution to the nonlocal equation (\ref{NLFP}) with a symmetric $J$ and
  $v$ the solution to the standard Fokker-Planck equation
  (\ref{eq:FP}) with the same initial condition, then
  \begin{equation*}
    \| u(t,\cdot) - v(t,\cdot) \| \leq C \e^2
    \qquad \text{for all $t \geq 0$,}
  \end{equation*}
  again in a variety of weighted $L^1$ norms. The constant $C$ is
  independent of $\e$ and $t$, so the convergence happens uniformly
  for all times. We refer to this as ``asymptotic-preserving
  convergence'', and is not proved in previous literature as far as we
  know.
\end{enumerate}
Let us describe these results more precisely. We always denote by
$L_1^k$ the usual weighted Lebesgue space with norm
$$\norm{\f}_{L^1_k}=\ird (1+|x|^k) |\f(x)| \d x$$ or with the
equivalent norm
$$\norm{\f}_{L^1_{\langle k\rangle}}=\ird \ap{x}^k |\f(x)| \d x,$$
where $\ap{x}^k:=(1+|x|^2)^{k/2}$, a ``smoothed $k$-th power'' which
avoids regularity problems at $x=0$. (We denote by $\ap{x}$ the
smoothed $1^{st}$ power.) Similarly we define $L^1_{\exp,a}$ as the
Lebesgue space with weight $e^{a|x|}$ with natural norm
$$
\norm{\f}_{\exp, a}=\ird e^{a|x|}\abs{\f(x)}\d x$$
or with equivalent norm
$$
\norm{\f}_{\ap{\exp}, a}=\ird e^{a\ap{x}}\abs{\f(x)}\d x.$$

An \emph{equilibrium} of equation (\ref{NLFP}) is a function $F_\e \in
L^1$ which satisfies the stationary PDE
\begin{equation*}
  \frac{1}{\e^2} \big( J_\e*F_\e - F_\e \big) + \div(xF_\e) = 0
\end{equation*}
in the distributional sense. A \emph{probability equilibrium} is a
nonnegative equilibrium $F_\e$ with $\ird F_\e(x) \d x = 1$. It is easy
to see that there exists at most one probability equilibrium (see
Theorem \eqref{eq: eqfourier}) and under mild additional conditions we
also show existence (and additional estimates). Here is the main
theorem we prove in this paper regarding asymptotic behavior of mild
solutions.
\begin{theorem}
  \label{spectral gap}
  Let $\e \in (0,1]$ and let $J \: \R^d \to [0,+\infty)$ satisfy
  (\ref{eq:J-moments-basic}) and, moreover, $J\in L^p$ for some
  $p\in(1,\infty]$. Then,
  \begin{enumerate} [(i)]
  \item \label{itemi} If $J \in L^1_{k'}$ for some $k'>2$, then there
    exists a unique equilibrium $F_\e$ of (\ref{NLFP}) in
    $L^1_{k'}$. Additionally, for every $0 < k \leq k'$ there exist
    (explicit) constants $C \geq 1$ and $\gamma > 0$, independent of
    $\e$, such that any (mild) solution $u$ of (\ref{NLFP}) with a
    probability initial condition $u_0 \in L^1_k$ satisfies
  \begin{equation*}
    \norm{u(t,\cdot)-F_\e}_{L^1_k}
    \le
    Ce^{-\gamma t}\norm{u_0-F_\e}_{L^1_k}
    \qquad \text{for every $t\ge 0$}.
  \end{equation*}
  
\item \label{itemii} If $J\in L^1_{\exp, a'}$, for some $a'>0$, then
  there exists a unique equilibrium $F_\e$ of (\ref{NLFP}) in
  $L^1_{\exp,a}$, for all $a<a'$. Additionally, there exist (explicit)
  constants $C \geq 1$ and $\gamma > 0$, independent of $\e$, such
  that any (mild) solution $u$ of (\ref{NLFP}) with a probability
  initial condition $u_0 \in L^1_{\exp, a}$ satisfies
  \begin{equation*}
    \norm{u(t,\cdot)-F_\e}_{L^1_{\exp,a}}
    \le
    Ce^{-\gamma t}\norm{u_0-F_\e}_{L^1_{\exp,a}}
    \qquad \text{for every $t\ge 0$}.
  \end{equation*}
  We remark that the unique equilibrium $F_\e$ also belongs to the
  (smaller, but $\e$-dependent) space $L^1_{\exp,b}$ for any
  $b< a' \e^{-1}$.
\end{enumerate}
\end{theorem}
The dependence of the constants can be explicitly traced throughout the proof. Among the various factors involved, the constants in general get better the closer $J$ resembles a Gaussian, in a sense that will be clarified later.

\begin{remark}\label{rmk: deacymass0} 
From our proof of the theorem (or directly as a
  simple consequence) we get the same estimates substituting
$u - F_\e$ by any solution $w$ with initial mass $0$ (i.e.
$\ird w_0(x) \d x = 0$). For example, we prove that under the
conditions of point (i),
\begin{equation*}
  \norm{w(t,\cdot)}_{L^1_k}
  \le
  Ce^{-\gamma t}\norm{w_0}_{L^1_k}
  \qquad \text{for every $t\ge 0$}.
\end{equation*}
This is indeed more general than our statement since one can take $w_0 = u_0
- F_\e$, but we choose to state the theorem as above for readability.
   
\end{remark}

We remark that both points (i) and (ii) of the previous theorem, and
all of our results concerning asymptotic behaviour, require in
particular that there exists $s\in (0,1]$ such that
\begin{equation}\label{eq:smoment}
  \ird J(x)|x|^{2+s}\d x=:\rho_{2+s}<\infty.
\end{equation}
This quantity $\rho_{2+s}$ will be fundamental for our estimates in
Section \ref{sec:asymptotic-behavior}.

Proving an estimate like the previous one has several
consequences. One of these is the following uniform-in-time result on
the closeness solutions to the local and nonlocal problems, which we
prove in Section \ref{sec:nonlocal-to-local}. For readability, we
state the theorem with likely non-optimal hypotheses, since optimizing
it is not the main goal of this paper. It requires the following
additional hypothesis on $J$:
\begin{equation}
  \label{eq:Jsym} 
  \ird J(x)x^\alpha\d x=0
  \qquad
  \text{for all multiindices $\alpha$ of order $|\alpha|=3$,}
\end{equation}
which is satisfied for example if $J$ is symmetric.

\begin{theorem}[Nonlocal-to-local approximation]
  \label{thm:asymptconv}
  Take $k\ge 0$ and assume $J\in L^1_{k+4+d}$ satisfies
  \eqref{eq:J-moments-basic} and \eqref{eq:Jsym}, $J \in L^p$ for some
  $p \in (1,\infty]$.  Let $u$ be the mild solution to equation
  (\ref{NLFP}) with a probability initial condition
  $u_0 \in L^1_k\cap \mathcal{C}^4$, and such that there exist
  $c > 0$ and $K>k+d$ with $|D^\alpha u_0(x)| \le c\ap{x}^{-K}$ for
  all multiindices $\alpha$ with $|\alpha|=4$.
  
  Let $v$ the (unique) integrable solution to the (local)
  Fokker-Planck equation (\ref{eq:FP}) with the same initial
  condition.  Then, there exists $C > 0$ such that
  \begin{equation}\label{eq: AP2}
    \norm{u(t, \cdot)-v(t, \cdot)}_{L^1_k}
    \leq C \e^2
    \qquad \text{for every $t\ge 0$ and $\e \in (0,1]$.}
  \end{equation}
  We emphasize that the constant $C$ is independent of $\e$.
\end{theorem}

\begin{remark}
  \label{rem:order-of-convergence}
  If we do not assume (\ref{eq:Jsym}), then under slightly weaker
  assumptions ($u_0\in L^1_k \cap \mathcal{C}^3$, $J\in L^1_{k+3 +d}$) one can show that
  \begin{equation}
    \label{eq: AP}
    \norm{u(t, \cdot)-v(t, \cdot)}_{L^1_k}
    \leq C \e
    \qquad \text{for every $t\ge 0$ and $\e \in (0,1]$,}
  \end{equation}
  for some (other) $\e$-independent constant $C > 0$. In order to obtain this it is enough to carry out the Taylor
  expansions in Section \ref{sec:nonlocal-to-local} to one order less.
\end{remark}

\begin{remark}\label{rmk: sec4exp} With the same strategy one
  can prove the previous statement also in Lebesgue spaces with
  exponential weight. Assume $J\in L^1_{\exp,a'}$ and
  $u_0\in L^1_{\exp,a'}\cap \mathcal{C}^3$. Then \eqref{eq: AP} holds also in
  $L^1_{\exp,a}$ for every $a<a'/2$. Moreover if \eqref{eq:Jsym}
  holds and $u_0\in \mathcal{C}^4$, then equation \eqref{eq: AP2} is also valid
  in $L^1_{\exp,a}$, for every $a<a'/2$. We don't give the proofs,
  since they are essentially the same as the ones in $L^1_k$
  spaces.
\end{remark}

Throughout the paper (as was used in the theorem above), we will use
multi-index notation for derivatives, so if $\alpha\in\mathbb{N}^d$ is
a multi-index we write
$|\alpha|=\sum_i\alpha_i$,
$x^\alpha=x_1^{\alpha_1} x_2^{\alpha_2}\cdots x_d^{\alpha_d}$,
$\alpha!=\alpha_1!\alpha_2!\cdots\alpha_d!$ and
$D^\alpha\f(x)=\partial_{x_1}^{\alpha_1}\partial_{x_2}^{\alpha_2}\dots\partial_{x_d}^{\alpha_d}\f(x)$.

We emphasize once more that, although being far from optimal, the
constants we find are all explicit, since proofs are completely
constructive.

\medskip

We will also make a few remarks regarding the unique probability
equilibrium $F_\e$ associated to (\ref{NLFP}). Theorems
\ref{spectral gap} and \ref{thm:asymptconv} imply in particular that
$\|F_\e- G\|_{L^1_k} \leq C \e^2$, where $G$ is the standard Gaussian
probability distribution in $\R^d$. But the properties of $F_\e$ are
quite interesting in their own right. Theorem \ref{spectral gap}
includes the fact that $F_\e$ has exponentially-decaying tails as long as $J$
has the same property. However, when $J$ is compactly supported we
suspect that $F_\e$ has tails which decay like the Poisson
distribution, that is $F_\e \approx e^{-a |x| \log |x|}$ for large
$|x|$. We are able to show an incomplete version of this statement in
Section \ref{sec:tails}. We suspect that in general $F_\e$ does
\emph{not} have Gaussian tails, but we haven't been able to prove
this.

The regularity of $F_\e$ is also interesting, as we explore in Section
\ref{sec:regularity}: independently of the regularity of $J$, $F_\e$
appears to be smooth everywhere except at $x=0$; and its regularity at
$x=0$ depends on $\e$, not on the regularity nor tails of $J$. The
equilibrium $F_\e$ becomes more regular as $\e \to 0$. We think a more complete
study of the properties of $F_\e$ would be interesting for a future work.

\paragraph{Strategy of the proof}

The central point of our results is the proof of \eqref{itemi} and
\eqref{itemii} in Theorem \ref{spectral gap}. The proof of Theorem
\ref{thm:asymptconv} can then be obtained as a consequence with fairly
standard estimates. As mentioned, the proof of Theorem \ref{spectral
  gap} uses Harris's theorem from probability, which requires us to
show a certain positivity property of solutions of our nonlocal
equation (\ref{NLFP}); see Section \ref{harristheory}, where we give a
short introduction to this. Showing that solutions are positive for
any \emph{fixed} $\e > 0$ is not hard; and the fact that
solutions to the (local) Fokker-Planck equation (\ref{eq:FP}) are
positive is well known, and a consequence of its regularization
properties and the parabolic Harnack inequality (or just a consequence
of its explicit solution!). However, finding a positivity result which
is valid for all $\e>0$ and which does not degenerate in the
limit $\e \to 0$ is not an easy task. We notice that it is
linked to convergence results in the central limit theorem, since the
solution $u$ to (\ref{NLFP}) can be written as (see Section
\ref{sec:representation})
\begin{equation*}
  u(t)(x)
  =
   e^{(d-\frac{1}{\e^2})t}
 \sum_{n=0}^\infty \Big(\frac{1}{\e^2}\Big)^n
  \int_0^t\int_0^{t_{1}}\dots\int_0^{t_{n-1}}\Jn*u_0(e^tx) \d
  t_n\dots \d t_1,
\end{equation*}
where it is understood that the term for $n=0$ is just $u_0(e^t x)$, and
\begin{equation*}
  \Jn:=J_{\e e^{t_1}}*\cdots*J_{\e e^{t_n}}(x).
\end{equation*}
Each of the terms represents the distribution of particles which have
undergone $n$ random jumps, if we see (\ref{NLFP}) as the PDE for the
distribution of a jump stochastic process. Proving positivity of $u$
can then be translated to the problem of understanding the behavior of
repeated convolutions of scalings of the function $J$, which can in
turn be understood in terms of the central limit theorem in
probability. This can be made into a rigorous proof, as we show in
this paper. We think it provides an intriguing link among these
different tools, and it yields a result which is otherwise non-obvious
to obtain.

\paragraph{Previous literature and related results}

We are not aware of any studies regarding the specific PDE
(\ref{NLFP}) except for \cite{mischler_uniform_2017}, whose results
are similar to our Theorem \ref{spectral gap}, but use very different
techniques (our conditions on $J$ are somewhat simplified, allowing
less regular and non-symmetric $J$). Our result in Theorem
\ref{thm:asymptconv} regarding the nonlocal-to-local limit as
$\e \to 0$ is new as far as we know. There are plenty of results on
related equations, of which we give a short summary now, but we
highlight that nonlocal-to-local limits which preserve the asymptotic
behavior in the limit are comparatively rare, and are a current source
of interest in several contexts.

The approach from nonlocal diffusion-type
operators of the form $J * u - u$ to local diffusion operators has
been the subject of many previous works. The PDE $\partial_t u = J*u - u$ is
studied in \cite{ChassChavesRossi2006}, and a review of related
techniques (making strong use of explicit solutions via the Fourier
transform) is given in \cite{AndreuVaillo2010}. Generalizations of
this result to similar nonlocal PDE with different integral kernels
are given in \cite{molino2019nonlocal}. The same PDE is studied in
\cite{Rey2013}, with a focus on its large-time behavior. In this case
the motivation comes from both physics (since this PDE is thought to
be a more appropriate diffusion model in some cases) and numerical
analysis (see below for more on this). Some results on its large-time
behavior can be obtained by using Fourier-type distances
\citep{Carrillo2007}, again by taking advantage of the fact that the
PDE is explicitly solvable in Fourier variables. A different nonlocal
and nonlinear version of the Fokker-Planck PDE (\ref{eq:FP}) is
studied in \cite{ignat2007nonlocal}. Its relationship to the local
problem is also studied there, but the results on asymptotic behavior
are not preserved when passing to the local limit.
Still a different nonlocal version has been recently considered in
\cite{auricchio2023}, motivated by a model in collective behavior.

We mentioned that equation (\ref{NLFP}) is a good test case for
ideas to study the asymptotic behavior of nonlocal PDE, and we would
like to describe a few more involved situations where this is
relevant. An equation close to (\ref{NLFP}) was proposed in \cite{cai2006stochastic} as a model for the probability of the concentration of a
certain protein which is being expressed by a particular gene. This
equation was later studied in \cite{Canizo2018b} by means of an entropy
inequality developed specifically for this PDE, which highlights that
a general strategy is lacking even for such simple nonlocal
equations.

Another important domain where these ideas are relevant is numerical
analysis. In dimension $d=1$, if we formally take
$J = \frac{1}{2} \left( \delta_{-1} + \delta_1 \right)$ then the
equation $\partial_t u = \frac{1}{\e^2} (J_{\e}*u -u)$, when evaluated at a
set $\{x_i\}_{i=0, \dots, N}$ of equally spaced nodes at distance
$\e$, becomes a numerical scheme for solving the standard heat
equation on $\R$. The proof in this paper does not apply directly to
this case, since this kernel does not have the $L^p$ regularity we
need. The question of whether a numerical method preserves the
large-time behavior of its limiting equation is of practical
importance, and it is a subject of recent research
\cite{ayi_structure-preserving_2022}, \cite{dujardin2020coercivity},
\cite{cances2020large} and \cite{bessemoulin2020hypocoercivity}.  We
sometimes refer to this as the property of the scheme being \emph{time
  asymptotic-preserving} (just ``asymptotic-preserving'' is often used
for schemes which are well-behaved under a certain scaling limit of
their associated PDE, which is not our case).  Numerical schemes for
PDE are nonlocal (in the sense that the equation at one of the nodes
depends on the value at other nodes, which are at a fixed positive
distance) and the study of their asymptotic behavior often encounters
similar difficulties as for nonlocal PDE. This was one of the
motivations of \cite{Rey2013}, and we have the same hope that
developing different methods for nonlocal PDEs can also suggest ideas
to understand the behavior of some numerical methods.

We finally mention that several methods to study the asymptotic
behavior of nonlocal PDE have been developed in the context of kinetic
theory. Some examples for linear PDE include semigroup tools and
spectral methods \citep{Gualdani2018}, techniques for hypocoercive PDE
\citep{villanihypo, dolbeauthypo}, applications of Harris's theorem
\citep{canizo2020hypocoercivity}, and entropy techniques, which are
often useful also for nonlinear PDE \citep{lods2008relaxation,
  Carrillo2007, bisi2015entropy}. Of these methods, the first one
based on spectral methods is the basis of
\cite{mischler_uniform_2017}; entropy techniques are the main tool of
\cite{Canizo2018b} on a very similar PDE; and as far as we
know our paper is the first time Harris-based techniques are used to
get the correct behavior as $\e \to 0$. It does not seem
straightforward to use entropy techniques to get a similar result,
since an inequality along the lines of \cite{Canizo2018b}
does not respect the scaling as $\e \to 0$, but finding a way to do so
is in our opinion an interesting problem.

\paragraph{Structure of the paper}

In Section \ref{sec:well-posed}, we present results related to the
well-posedness of the equation \eqref{NLFP} in $L_k^1$. For this
purpose we define the \textit{mild solutions} of \eqref{NLFP} by a
standard use of Duhamel's formula, for which we prove existence,
uniqueness and continuous dependence on the initial data via the well
known fixed-point strategy. We point out that the same strategy can be
used to prove well-posedness in measure spaces, and interpret these
results also in terms of semigroups.

In the last part of Section \ref{sec:well-posed} we give two explicit
representations of the solution: a classical one via Fourier
transform, and the Wild sum representation which we use in a crucial
way.

Section \ref{sec:asymptotic-behavior} contains the proof of Theorem
\ref{spectral gap}. We give a short introduction to Harris's Theorem,
which requires two main conditions: a Lyapunov confinement condition
which is fairly straightforward in our case, based on a Taylor's
expansion; and a positivity condition which is the key part of our
argument, as we discussed above. Proving this positivity condition
requires a suitable version of the Berry-Esseen Theorem, a form of the
central limit theorem which gives a rate of convergence in the
$L^\infty$ norm. We need a result for independent, not necessarily
identically distributed, random vectors. Afterwards, the proof of
positivity is straightforward once a suitable scaling is carried out.

Section \ref{sec:nonlocal-to-local} deals with the speed of
convergence of the Nonlocal Fokker Planck equation (\ref{NLFP}) to the
local on (\ref{eq:FP}) at all time $t$, and gives the proof of Theorem
\ref{thm:asymptconv}. The rate of convergence at finite times comes
from a standard consistency and stability argument in the spirit of
numerical methods. The latter result, combined with Theorem
\ref{spectral gap}, allows us to extend the convergence for all times.

In the last section \ref{sec: moments} we gather some independent
results on the regularity of the equilibrium, showing that it improves as $\e\to 0$. Moreover, we give a representation of the solution of \eqref{NLFP}
in terms of the cumulative generating function, which is nothing else
than the logarithm of the moment generating function. This gives in
particular an explicit representation for the moments of the
equilibrium. We also investigate on the tails of the solution of
\eqref{NLFP} and its equilibrium $F_\e$. In the local case, the
equilibrium is a Gaussian and the solution has Gaussian tails if the
initial data has. For the nonlocal equation (\ref{NLFP}) we strongly
suspect this is not the case, even for very nice initial data and
kernel. We are able to prove that for compactly supported $J$ the
tails are at least Poisson-like.

\section{Well-posedness of the initial value problem}
\label{sec:well-posed}

We gather in this section some results on the well-posedness of
equation (\ref{NLFP}), that is,
\begin{equation*}
  \partial_t u = \frac{1}{\e^2} \big(J_\e*u-u\big) + \div(xu),
\end{equation*}
with the addition of an initial condition $u_0 \in L^1_k$.
\begin{equation*}
  u(0,x) = u_0(x), \qquad x \in \R^d.
\end{equation*}
Though they are mostly standard results for linear PDE, it is useful
to lay out the basic theory.

\subsection{Existence and Uniqueness}
Solutions to equation \eqref{NLFP} can be studied (for example) by
using semigroup theory, which easily yields mild solutions
$u \in \mathcal{C}([0,\infty), L_k^1(\mathbb{R}^d))$ whenever
$u_0 \in L^1_k$, for some $k \geq 0$. Alternatively, we give a
direct simple proof of existence of mild solutions. The first
observation is that the equation
\begin{equation*}
  \partial_t u = \div(xu)-\frac{1}{\e^2} u, \qquad u(0,x) = u_0(x)
\end{equation*}
has the explicit classical solution
\begin{equation*}
  u(t,x) = e^{(d-1/\e^2)t} u_0(e^t x), \qquad t \geq 0, x \in \R^d,
\end{equation*}
at least for a differentiable $u_0$. Hence, for any $u_0 \in L^1$
we define
\begin{equation}
  \label{eq:St}
  T_t[u_0] (x) \equiv T_t u_0(x) := e^{(d-1/\e^2)t} u_0(e^t x)
  \qquad \text{for $t \geq 0$, $x \in \R^d$.}
\end{equation}
Following a well known strategy, we use Duhamel's formula to define
mild solutions of equation \eqref{NLFP} as follows:

\begin{definition}[Mild solution]
  \label{dfn:mild}
  Let $k \geq 0$, and take $J$ satisfying Hypothesis
  \eqref{eq:J-moments-basic} on $J$ and $J\in L^1_k$. Take
  $u_0 \in L^1_k$, and take $T^* \in (0,+\infty]$.  A \emph{mild
    solution} of equation \eqref{NLFP} on $[0,T^*)$ with initial
  condition $u_0$ is a function $u \in \mathcal{C}([0,T^*), L^1_k)$ such
  that
  \begin{equation*}\label{mild1}
    u(t) = T_t u_0 +  \frac{1}{\e^2}\int_0^t T_{t-s}
    \big[ J_\e*u(s)\big] \d s
    \qquad
    \text{for all $t \in I$,}
  \end{equation*}
  where the integral is understood (equivalently) in the Riemann or
  Bochner sense.
\end{definition}

\subsubsection{Existence of $L^1_k$ solutions}

We are going to show the following result:

\begin{theorem}
  \label{thm:wellposed}
  Take $k \geq 0$ and let $J$ satisfy
  \eqref{eq:J-moments-basic} and, in addition, $J \in L_1^k$. Then,
  there exists a unique mild solution of equation \eqref{NLFP} with
  initial condition $u_0\in L^1_k$ (in the sense of Definition
  \ref{dfn:mild}). In addition, the solution is continuously dependent
  on the initial data, in the following sense: for any two mild solutions $u$
  and $v$ of \eqref{NLFP} with initial data $u_0$ and $v_0$
  \begin{equation*}
    \norm{u(t)-v(t)}_{L^1_k}\le e^{\frac{t}{\e^2}(c_k-1)}\norm{u_0-v_0}_{L^1_k}
  \end{equation*}
\end{theorem}

The proof of this is fairly straightforward, and we give it after a
couple of lemmas. We first notice that the convolution operator
$u \mapsto J * u$ is a continuous operator under fairly mild
conditions. For any measurable functions $u \: \R^d \to \R$ and
$V\: \R^d \to [1,+\infty)$ we denote
\begin{equation*}
    \| u\|_V := \ird |u(x)| V(x) \d x.
\end{equation*}
We need the following generalization of Young's inequality in weighted
spaces, whose proof we omit since it is obtained by a direct
computation using the subadditivity in the assumptions:
 \begin{lemma}\label{lem:weighted_young}
  Take a measurable $J \: \R^d \to \R$ such that $\|J\|_V\|u\|_V <
  +\infty$. Then for every sub-additive function
  $V \: \R^d\to [1,\infty)$ (that is, satisfying for all $x, y \in \R^d$
  $V(x+y) \le C(V(x)+V(y))$ for some $C > 0$ )
  it holds that
  $$\norm{J*u}_V \le 2C\norm{J}_V\norm{u}_V$$
  for all measurable $u \: \R^d \to \R$ such that $\|u\|_V <
  +\infty$. In particular, for $V(x) := 1 + |x|^k$ with $k \geq 0$ we have
  \begin{equation}
    \label{eq:Lk-Young}
  \norm{J*u}_{L^1_k} \le 2^{k} \norm{J}_{L^1_k} \norm{u}_{L^1_k}.
  \end{equation}
\end{lemma}
We also notice that for any $t \geq 0$, the operator $T_t$ is
contractive in all $L^1_k$ for any $k \geq 0$:
\begin{lemma}[Contractivity]
  \label{lem:St_contractive}
  For all $k\geq 0$, $t>0$ and $u \in L^1_k$, the operator $T_t$
  defined in \eqref{eq:St} satisfies
   \begin{equation*}
    \| T_t u \|_{L^1_k} \le e^{-t/\e^2}\| u \|_{L^1_k}.
  \end{equation*}
\end{lemma}
\begin{proof}
  We just write the following direct estimate:
  \begin{multline*}
      \| T_tu \|_{L^1_k} =
      \int (|x|^k+1) |T_t u(x)| \d x
      =
      e^{-t/\e^2}\int(|x|^k+1)|u(xe^t)|e^{d t} \d x
      \\
      =
      e^{-t/\e^2}\int (|e^{-t}y|^k+1)|u(y)| \d y
      <
      \int (|y|^k+1)u(y) \d y.
  \end{multline*}
\end{proof}%
\begin{proof}[Proof of Theorem \ref{thm:wellposed}]
  This is done via a standard contraction argument.  Take $T^* > 0$
  and define the Banach space
  $\mathcal{Y} := \mathcal{C}([0,T^*],L^1_k)$ equipped with the norm
  $$\norm{u}_{\mathcal{Y}}:= \sup_{t\in [0,T^*]}\norm{u(t)}_{L_k^1}.$$
  Let us define $\Psi:\mathcal{Y}\rightarrow\mathcal{Y}$ by
  \begin{equation*}
    \Psi[u](t):=T_tu_0+\frac{1}{\e^2}\int_0^t T_{t-s}\big[ J_\e*u(s)\big] \d s    
  \end{equation*}
  for any $u \in \mathcal{Y}$, so that $u$ is a mild solution if and
  only if $u$ is a fixed point for $\Psi$. In the following estimate $u_s$ and $v_s$ will be used as shorthand for $u(s)$ and $v(s)$. For every $u_t,v_t\in
  L^1_k$, use Lemmas \ref{lem:weighted_young} and
  \ref{lem:St_contractive} we have 
  \begin{equation*}
    \begin{split}
      \|\Psi[u](t)&-\Psi[v](t)\|_{L^1_k}
      =
      \left\|
      \frac{1}{\e^2}\int_{0}^tT_{t-s}\big[J_\e*(u_s-v_s) \big] \d s
      \right\|_{L^1_k}
      \\
      & \le 
      \frac{1}{\e^2}\int_{0}^t \left\| T_{t-s} \big[\big|J_\e*(u_s-v_s)\big|
        \big] \right\|_{L^1_k} \d s
        \\
      &\le
      \frac{1}{\e^2}\int_{0}^t e^{-\frac{t-s}{\e^2} }\left\|J_\e*(u_s-v_s)
      \right\|_{L^1_k} \d s
    \\
      &\le\frac{1}{\e^2} c_k \int_{0}^t \|u_s-v_s\|_{L_k^1} \d s 
      \\
      &\le \frac{1}{\e^2}c_k t \,\, \|u-v\|_{\mathcal{Y}},
    \end{split}
   \end{equation*}
   where $c_k = 2^k\norm{J}_{L^1_k} \geq 2^k\norm{J_\e}_{L^1_k}$ 
   from Lemma \ref{lem:weighted_young}.
   Iterating the process, we obtain
   \begin{equation*}
       \|\Psi^n[u](t)-\Psi^n[v](t)\|_{L^1_k}
     \le
      \Big(\frac{c_k t}{\e^2}\Big)^n\frac{1}{n!}\|u-v\|_{\mathcal{Y}}\le \Big(\frac{c_k T^*}{\e^2}\Big)^n\frac{1}{n!}\|u-v\|_{\mathcal{Y}}
   \end{equation*}
and then, taking the supremum, leads to   
   \begin{equation*}
   \begin{split}
     \|\Psi^n[u](t)-\Psi^n[v](t)\|_{\mathcal{Y}}
     \le
      \Big(\frac{c_k T^*}{\e^2}\Big)^n\frac{1}{n!}\|u-v\|_{\mathcal{Y}}
      \end{split}
   \end{equation*}
   For every $T^*$, there exists $n\in \mathbb{N}$ such that
    $\big(\frac{c_k T^*}{\e^2}\big)^n\frac{1}{n!}< 1$. 
   There is a specific version of the contraction mapping theorem (see
   for example \citet[Thm. 1.1.3]{miklavcic_applied_1998}) which now
   implies $\Psi$ has a fixed point which is a solution of
   \eqref{NLFP} in the sense of Definition \ref{dfn:mild}.
   (Alternatively, one can consider the sequence
   $(\Psi^n(u_0))_{n \geq 0}$, where $u_0 \in \mathcal{Y}$ is the function
   constantly equal to $u_0$ on $[0,T^*]$; our last inequality easily
   shows that it is a Cauchy sequence, and hence must have a
   limit. This can be done for any $T^* > 0$, easily obtaining the
   result on $[0,+\infty)$.)

   To prove continuity with respect to initial data we proceed as
   before: let $v$ be another mild solution for some other initial
   data $v_0$, then for every $t$
    \begin{equation*}\begin{split}
       &\norm{u(t)-v(t)}_{L^1_k}\le \norm{T_t(u_0-v_0)}_{L^1_k}
       +\frac{1}{\e^2}\int_0^t\norm{
         T_{t-s}[J_\e*(u_s-v_s)]}_{L^1_k} \d s
       \\
       &\le e^{-t/\e^2}
       \norm{u_0-v_0}_{L^1_k}
       +\frac{c_k}{\e^2}\int_0^t
       e^{-\frac{t-s}{\e^2}}\norm{u(s)-v(s)}_{L^1_k}  \d s.
     \end{split}
   \end{equation*}
   By the standard Gronwall lemma applied to
   $e^{\frac{t}{\e^2}} \norm{u(t)-v(t)}_{L^1_k}$ we obtain the result.
 \end{proof}

\subsubsection{Existence of measure solutions}

We sketch similar arguments that give existence of solutions in a
space of measures, since  Harris's theorem we use later is usually
stated for semigroups in a space of measures (or Markov semigroups in
the context of probability). This is really not strictly necessary,
since there are versions of the Harris theorem which hold similarly in
weighted $L^1$ spaces (see \cite{canizo_harris-type_2021}), but since
the proof is very similar we give the main steps.

Let us denote with $\M(\R^d)$ the space of signed finite Radon
measures over $\R^d$ and by $\norm{\cdot}_{TV}$ and
$\norm{\cdot}_{\M_k}$ the total variation norm and the weighted total
variation, respectively. Let us recall the definition of Bounded
Lipschitz norm, i.e.
\begin{equation*}
   \norm{\mu}_{BL}:=\sup_{f\in\mathcal{F}}\abs{\ird f\d\mu}
\end{equation*}
where
$\mathcal{F}:=\{f \in W^{1,\infty}(\R^d):\norm{f}_{1,\infty}\le 1\}$,
with $\norm{f}_{1,\infty} := \|f\|_\infty + \Lip(f)$ and $\Lip(f)$
denoting the Lipschitz constant of $f$.
 In the literature,  the bounded Lipschitz norm is sometimes defined as the equivalent norm obtained by choosing in the previous definition
$\mathcal{F}:=\{ f:\R^d \to\R:\norm{f}_\infty\le1, \sup_{x\ne
  y}\frac{\abs{f(x)-f(y)}}{|x-y|}\le 1 \}$.

We define the ball $\mathcal{X}_R:=\{\mu\in\M:\norm{\mu}_{\M_k}\le R\}$ and the
set $\mathcal{Z}=\mathcal{C}([0,T^*], \mathcal{X}_R)$ of functions which are continuous
from $[0,T^*]$ to $\mathcal{X}_R$ in the bounded Lipschitz norm. $\mathcal{Z}$
is a complete metric space with the metric induced by the norm
$$\norm{\mu}_{\mathcal{Z}}= \sup_{t\in[0,T^*]} \norm{\mu(t)}_{BL}.$$
We stress that the same claim in the space $\mathcal{C}([0,\infty), \M_k)$ does
not hold since in $(W^{1,\infty})^*$ there are elements that are not
measures which can be approximated by Cauchy sequences in $\M_k$.

Following the path of the previous proof, we can prove existence and
uniqueness of the mild solution of equation \eqref{NLFP} in
$\mathcal{Z}$ choosing a suitable $R$ such that $\mu_0 \in \mathcal{X}_R$. One
can prove a version of inequality (\ref{eq:Lk-Young}) in weighted
total variation norms:
\begin{equation*}
  \norm{J*\mu}_{\M_k}\le c_k \norm{\mu}_{\M_k}
\end{equation*}
Moreover, the contractivity Lemma \ref{lem:St_contractive} still holds
in the sense of measures in weighted total variation norm.
 The functional $\Psi$, defined as \begin{equation*}
        \Psi[\mu](t)= T_t\mu_0+\frac{1}{\e^2}  \int_0^tT_{t-s}[J_\e*\mu_s]\d s
     \end{equation*}
     is such that if $t\to\mu(t)\in \mathcal{Z},$ then so is
     $t\to\Psi[\mu](t)$ for small enough $T^*$. We briefly show a
     proof of this.  Continuity comes straightforwardly from the fact
     that $T_t$ is a continuous semigroup (in the bounded Lipschitz
     norm) and all the other operations are continuity-preserving.
     
     We have left to show that $\Psi[\mu](t)\in \mathcal{X}_R$ for every
     $t\in[0,T^*]$ when $\mu(t)\in\mathcal{X}_R$. Set $R=2\norm{\mu_0}_{\M_k}$
\begin{equation*}
\begin{split}
   &\norm{\Psi[\mu](t)}_{\M_k}\le\norm{T_t\mu_0}_{\M_k}+\frac{1}{\e^2}\int_0^t\norm{T_{t-s}[J_\e*\mu_s]\d s}_{\M_k}\\
    &\le \norm{\mu_0}_{\M_k}+\frac{1}{\e^2}\int_0^tC_k\norm{\mu_s}_{\M_k}\d s\le\norm{\mu_0}_{\M_k}+\frac{1}{\e^2}CT^*\norm{\mu}_{\mathcal{\M}_k}\le R
    \end{split}
\end{equation*}
for $T^*\le\frac{\e^2}{2C_k}$. We then iterate the process for another time $T^*_1$ until we extend the existence and uniqueness globally.

With the same approach, one can prove the previous results in the space $L^1_{\exp,a}$. Since the procedure is the same \textit{mutatis mutandis}, we have chosen to omit it.
\subsection{Semigroup results}

Our result on the well-posedness of the nonlocal Fokker-Planck
equation \eqref{NLFP} can also be phrased in terms of semigroup
theory. For an introduction and further details on the results we will
be using below, we refer the reader to the books \cite{EngelNagel2001}
or \cite{pazy_semigroups_1983}. We first point out that for any
$k \geq 0$, the operators $(T_t)_{t \geq 0}$ given in \eqref{eq:St}
define a strongly continuous semigroup on the Banach space
$L^1_k$; it is easy to check this directly. The generator of
this semigroup is the operator $B_\e$ given by
\begin{equation*}
  B_\e u := \div(x u)-\frac{1}{\e^2}u
\end{equation*}
with domain
\begin{equation*}
  \mathcal{D}(B_\e) := \{ u \in L^1_k \mid \div(x u)\in L^1_k \},
\end{equation*}
which is a closed linear operator.
Then,  equation  \eqref{NLFP} can be written as
\begin{equation}
  \label{nlfp_S}
  \ddt u = L_\e u = A_\e u + B_\e u,
\end{equation}
where $L_\e$ and $A_\e$ denote the operators
\begin{gather*}
  A_\e u := \frac{1}{\e^2} J_\e*u ,
  \\
  L_\e u := A_\e u + B_\e u
  = \frac{1}{\e^2} \left( J_\e*u - u \right) +
  \div(x u).
\end{gather*}
Due to Lemma \ref{lem:weighted_young} we know that
$A_\e \: L^1_k \to L^1_k$ is a bounded linear
operator. Since we know $B_\e$ is a generator in $L^1_k$, standard
results for bounded perturbations in semigroup theory (see for example
\citet[chapter III.1]{EngelNagel2001}) ensure then that the operator
$L_\e=A_\e + B_\e$ is also a generator with domain $\mathcal{D}(L_\e)=\mathcal{D}(B_\e)$. This is equivalent to the
well-posedness of equation \eqref{nlfp_S} in the sense of semigroups
\citep[Chapter II.6]{EngelNagel2001}. Solutions to \eqref{nlfp_S}
in the sense of semigroups must also be mild solutions in the sense of
Definition \ref{dfn:mild} \citep[Section III.1.3, Corollary
1.4]{EngelNagel2001}, so they coincide with the solutions given in
Theorem \ref{thm:wellposed}.
\subsection{Representation of the solution}
\label{sec:representation}
For this kind of problems, we can express the solution via the classic
Fourier transform or with the so-called Wild sum. We use the latter as
an important tool in our later proofs, emphasizing that our main
results do not seem easy to obtain by using the explicit
representation in the terms of the Fourier transform. More
importantly, we would like to explore strategies which would be valid
even in cases in which Fourier transform methods are not
available. 

For the sake of completeness and to obtain regularity results
  in section \ref{sec:regularity}, we will also briefly comment on
the Fourier solution. From
\eqref{eq:J-moments-basic} we deduce that
\begin{equation*}\label{eq: Jfourier}
  \hat{J}(\xi)=1-\xi^2+o(|\xi|^2)
  \qquad \text{as $\xi \to 0$.}
\end{equation*}

Then if we take $u_0$ as an initial probability distribution
  in $L^1$, we can prove that there exists at most one solution of \eqref{NLFP} in
  $\mathcal{C}([0,\infty),L^1)$. If it exists, it must satisfy
\begin{equation}\label{eq: solfourier}
    \hat{u}(t,\xi)=\hat{u}_0(e^{-t}\xi)e^{\frac{1}{\e^2}\int_0^t\zeta_\e(e^{-s}\xi)\d s}
\end{equation}    
with $\zeta_\e(\xi)=\hat{J}_\e(\xi)-1$.  Moreover, there exists at
most one equilibrium in $L^1$ of \eqref{NLFP}, and it must satisfy
\begin{equation}\label{eq: eqfourier}
  \hat{F}_\e(\xi)=e^{\frac{1}{\e^2}\int_0^\infty\zeta_\e(e^{-s}\xi)\d s}
\end{equation}
 Indeed, passing to Fourier variables, equation \eqref{NLFP} becomes
\begin{equation*}
    \partial_t\hat{u}=\frac{1}{\e^2}\big(\hat{J}_\e\hat{u}-\hat{u}\big)-\xi\cdot \nabla_\xi\hat{u}=\frac{1}{\e^2}\zeta_\e(\xi)\hat{u}-\xi\cdot \nabla_\xi\hat{u}
\end{equation*}
which has equation \eqref{eq: solfourier} as its unique solution.  We
also notice that the transform of the equilibrium must satisfy
 \begin{equation*}
       \hat{F_\e}(\xi)= \hat{u}(r,\xi)=\hat{F}_\e(e^{-r}\xi)e^{\frac{1}{\e^2}\int_0^r\zeta_\e(e^{-s}\xi)\d s}\quad\text{ for every $r\ge 0$}
   \end{equation*}
   with the initial condition  $\hat{F}_\e(0)=1$, since $\norm{F_\e}_{L^1}=1$.
   In particular
   \begin{equation*}
     \hat{F}_\e(\xi)=\lim_{r\to\infty}\hat{F}_\e(e^{-r}\xi)e^{\frac{1}{\e^2}\int_0^r\zeta_\e(e^{-s}\xi)\d
       s}=e^{\frac{1}{\e^2}\int_0^\infty\zeta_\e(e^{-s}\xi)\d s}.
     \qedhere
   \end{equation*}

\begin{remark}
  With a more careful study of these arguments or other classical ones
  (like Tychonoff's Theorem) one could prove existence of a solution
  and of the equilibrium. We don't want to insist on these details
  since in the next section we will give a different proof for the
  existence and uniqueness of the equilibrium.
\end{remark}

\paragraph{Wild Sum}

When dealing with linear convolution-type operators, Wild Sum---or
Dyson-Phillips 's Series \cite[p. 163 ]{EngelNagel2001}---give a
particularly simple way to express the solution, based on a ``Picard
iteration'' for PDEs.

We first present the following lemma that, for a fixed $t$, allows us
to exchange the order of the operator $T_t$ with the convolution
operator and vice versa.
\begin{lemma}\label{interchange}
Let $(T_t)$ be the semigroup associated to the solution of equation \eqref{eq:St}. Then
\begin{equation*}
    J_\e*(T_tf)=T_t[J_{\e e^t}*f]
\end{equation*}
\end{lemma}
\begin{proof}
  For every $x$ 
  \begin{multline}
            \big(J_\e*T_tf\big)(x) =\int\e^{-d}J\Big(\frac{x-y}{\e}\Big)(T_tf)(y) \d y=e^{(d-\frac{1}{\e^2})t}\int \e^{-d}J\Big(\frac{x-y}{\e}\Big)f(e^ty) \d y\\
            =e^{-\frac{t}{\e^2}}\int\e^{-d} J\Big(\frac{e^{-t}}{\e}(e^tx-z)\Big)f(z)\d z\\
            =\int e^{(d-\frac{1}{\e^2})t}\big(e^{-d t}\e^{-d}\big)J\Big(\frac{e^tx-z}{e^t\e}\Big)f(z)\d z=e^{(d-\frac{1}{\e^2})t}\big(J_{\e e^t}*f\big)(e^tx)
    \end{multline}
    where we have used the change of variable
    $z=e^ty$. Hence,
    \begin{equation*}
    J_\e*T_tf=T_t[J_{\e e^t}*f]. \qedhere
\end{equation*}
\end{proof}

\begin{theorem}\label{Thm:wild}
  Let $u$ be the unique solution to equation \eqref{NLFP} with initial condition $u_0\in L^1$.  Then, in
  the hypotheses of Theorem \ref{thm:wellposed}, we can write the
  solution of the equation via Wild sum, namely for every $t\ge 0$
  
  \begin{equation*}
    u(t)(x)
    =
    e^{(d-\frac{1}{\e^2})t}\Big[u_0(e^tx)
    +\sum_{n=1}^\infty \Big(\frac{1}{\e^2}\Big)^n
    \int_0^t\int_0^{t_{1}}\dots\int_0^{t_{n-1}}\Jn*u_0(e^tx)
    \d t_n\dots \d t_1.
    \Big]
\end{equation*}
where the convergence in the RHS is understood to be in $L^1$ for every $t$, and
\begin{equation*}\label{jn}
    \Jn:=J_{\e e^{t_1}}*\cdots*J_{\e e^{t_n}}.
\end{equation*}
for $0\le t_1\le\dots\le t_n\le t$.
\end{theorem}

\begin{proof}
Define \begin{equation*}
\begin{split}
    h(t):=T_{-t}u(t)&=T_{-t}\Big[ T_tu_0+\int_0^t T_{t-s}[\frac{1}{\e^2}J_\e*u(s)]\d s \Big]\\
    &= u_0+\frac{1}{\e^2}\int_0^t T_{-s}[J_\e*u(s)]
    \end{split}
\end{equation*}
Using Lemma \ref{interchange} 
\begin{equation*}
        T_{-s}[J_\e*u_s]=T_{-s}[J_\e *T_s(T_{-s}u)]=T_{-s}[T_sJ_{\e e^s}*h(s)]=J_{\e e^s}*h(s)
\end{equation*}
This means that $ u$ is a mild solution for Equation \eqref{NLFP} iff
$$h(t)=u_0+\frac{1}{\e^2}\int_0^t J_{\e e^s}*T_{-s}u(s)\d s=u_0+\frac{1}{\e^2}\int_0^t J_{\e e^s}*h(s)\d s,$$ that is 
if $h$ is a fixed point for the functional  $\Psi$
\begin{equation*}
    \Psi(g):=u_0+\frac{1}{\e^2}\int_0^t J_{\e e^s}*g_s\d s
\end{equation*}
To find an expression for this fixed point we perform a Picard Iteration defining        $h_0=u_0$ and 
\begin{equation}\label{picard}
         h_{n+1}=\Psi(h_{n})
         \end{equation}
         i.e.
         $$h_{n}(t) = h_{n-1}(t) + \frac{1}{\e^{2n}}\int_0^t\int_0^{t_{1}}\dots\int_0^{t_{n-1}} J_\e^{t_1,\dots,t_n}*u_0\d t_n\dots \d t_1.$$
Clearly the sequence $(h_n)_n$ is pointwise monotone increasing. Since
{the functions $h_n$ are all nonnegative and bounded
  above by $h$}, the monotone convergence theorem shows that $h_n$
converges in $L^1$ to some function $\tilde{h} = \tilde{h}(t)\ge0$ when
$n\rightarrow \infty$, namely
\begin{equation*}
    \tilde{h}(t) = u_0+\sum_{n=1}^\infty \Big(\frac{1}{\e^2}\Big)^n\int_0^t\int_0^{t_{1}}\dots\int_0^{t_{n-1}} J_\e^{t_1,\dots,t_n}*u_0\d t_n\dots \d t_1.
\end{equation*}
Passing to the limit in \eqref{picard}, one obtains that $\tilde{h}$
is a fixed point for the functional $\Psi$, so in fact
$\tilde{h} = h$. Therefore,
\begin{align*}
  &u(t)(x)= \big(T_th\big)(x)=e^{(d-1/\e^2)t}u_0(e^tx)+
  \\
  &+e^{(d-1/\e^2)t}\sum_{n=1}^\infty
    \e^{-2n}\int_0^t\int_0^{t_{1}}\dots\int_0^{t_{n-1}}\Big(J_\e^{t_1,\dots,t_n}*u_0\Big)(e^tx)
    \d t_n\dots \d t_1.
    \qedhere
\end{align*} 
\end{proof}

\begin{remark}\label{rmk: propinh}
  From the previous it  is immediate to see that if $u_0$ and $J$ are
  radially symmetric functions then also the solution is radially
  symmetric. One can also say this of any property which is preserved
  by convolution with $J$. Analogously, the equilibrium will preserve
  the same properties.
\end{remark}

\section{Asymptotic behavior}
\label{sec:asymptotic-behavior}

In this section we study the asymptotic behavior of equation
\eqref{NLFP}. For this we use Harris's theorem from probability,
establishing exponential convergence to equilibrium of solutions in
$\M_k$ for $k > 0$. Our crucial advantage of our approach is that it
yields a uniform bound on the convergence rate to equilibrium for all
$\e \in (0,1]$.

\subsection{Harris's Theorem }\label{harristheory}

Let us first give a short overview of Harris's theorem, in order to
make this paper as self-contained as possible. We refer again to
\cite{Hairer2011, canizo_harris-type_2021} and the book by
\cite{Meyn2010} for a more complete exposition.

Let us denote by $\M \equiv \M(\Omega)$, the space of finite signed
measures over a measurable space $\Omega$, with
$\norm{\mu}=\norm{\mu}_{TV} = \int |\mu|$, the total variation
norm. $\mathcal{P}$ will denote the set of probability measures in
$\M$. If $V:\Omega\rightarrow [1,\infty)$ is a measurable function
(which we call a \emph{weight function}), we denote by $\M_V$ the
subspace of measures $\mu \in \M$ such that the norm
$\norm{\mu}_{V}=\int V|\mu|<\infty$. We will use $\mathcal{P}_V$ to
denote $\M_V\cap\mathcal{P}$.

A linear operator $S:\M\to\M$ is \emph{stochastic} if it leaves $\mathcal{P}$
invariant. A stochastic semigroup $(S_t)_{t\ge0}$ is a family of
stochastic operators $S_t$ parameterised by $t$ such that $S_0=\Id$
and $S_t\circ S_s=S_{t+s}$ for all $t, s \geq 0$.

We say that $S$ is a stochastic operator on $\M_V$ if it is a
stochastic operator on $\M$ and its restriction on $\M_V$ is bounded
in the $\norm{\cdot}_V$ norm. Similarly, $S_t$ is a stochastic
semigroup on $\M_V$ if it is a stochastic semigroup on $\M$ and each
$S_t$ is a stochastic operator on $\M_V$.

A basic tool in the study of the asymptotic behavior of Markov
Processes is Doeblin's Theorem that under a (strong) positivity
hypothesis gives existence of a unique stationary state, which is
additionally exponentially stable.  A well known generalization of it
is Harris's Theorem where Doeblin's condition is replaced by a weaker
local positivity condition and a Lyapunov confinement, which we state
next.

\begin{hyp}[Semigroup Lyapunov condition]
  \label{ass: semigroupLyap}
  Let $V \: \Omega\to[1,\infty)$ be a measurable (``weight'') function
  and let $(S_t)_{t\ge0}$ be a stochastic semigroup on $\M_V$. We say
  that $(S_t)_{t\ge 0}$ satisfies a semigroup Lyapunov condition if
  there exist constants $\lambda$, $C>0$ such that
  \begin{equation}
    \label{eq: semigroupLyap}
    \tag{SLC}
    \norm{S_t \mu}_{V}\le e^{-\lambda t}\norm{\mu}_V+\frac{C}{\lambda}
    (1-e^{-\lambda t})\norm{\mu}.
  \end{equation}
\end{hyp}
If $S_t$ is a semigroup with generator $L$, a natural and better known
version of this condition is to impose
\begin{equation}
  \label{eq:sLyap}
  L^*V\le C-\lambda V
\end{equation}
where $L^*$ is the (formal) dual operator of $L$ acting on ``nice
enough'' functions (see for example \cite{hairer2010convergence},
where inequality (\ref{eq:sLyap}) must be understood in a
submartingale sense). We show condition (\ref{eq:sLyap}), in a
suitable sense, for the generator associated to equation (\ref{NLFP})
in Theorems \ref{thm:lyapunovpow} and \ref{thm:lyapunovexp}. From this
we will deduce Hypothesis \ref{ass: semigroupLyap}.

The second condition needed is the \emph{Harris condition}:

\begin{hyp}[Harris condition]
\label{HLB}
A stochastic operator $S$ on $\M$ satisfies a Harris condition on a
set $\mathcal{S} \subseteq \Omega$ if there exist $0<\alpha<1$ and
$\rho \in \mathcal{P}$ such that
\begin{equation}\label{eq:HLB}\tag{HC}
    S\mu\ge \alpha\rho\int_{\mathcal{S}}\mu\quad \text{ for every } 0\le\mu\in\M.
\end{equation}
\end{hyp}

We will use the following version of Harris's Theorem:

\begin{theorem}[Harris's theorem for semigroups.]
  \label{sHarris}
  Let $V:\Omega\to[1,\infty)$ be a measurable function and let
  $(S_t)_{t\ge0}$ be a stochastic semigroup on $\M_V$. Assume that:
  \begin{enumerate}
  \item the
    semigroup $(S_t)_{t \geq 0}$ satisfies the semigroup Lyapunov
    condition \eqref{eq: semigroupLyap}, and
    
  \item for some $T^* > 0$ and some $R>\frac{2C}{\lambda}$, the
    operator $S_{T^*}$ satisfies the Harris condition \eqref{eq:HLB}
    on the set $\mathcal{S}:=\{x:V(x)\le R \}$.
  \end{enumerate}
  Then the semigroup has an invariant probability $\mu^*\in\mathcal{P}_V$
  which is unique in $\mathcal{P}_V$, and there exist $C,\lambda>0$ such that
  \begin{equation*}
    \norm{S_t\nu}_V\le Ce^{-\lambda t}\norm{\nu}_V
    \quad\text { for all $\nu\in\mathcal{P}_V$ with $\int_\Omega \nu =
      0$ and all $t\ge0$},
  \end{equation*}
  and in particular
  \begin{equation*}
    \norm{S_t\mu-\mu^*}_V\le Ce^{-\lambda t}\norm{\mu-\mu^*}_V
    \quad\text { for all }\mu \in \mathcal{P}_V \text{ and }t\ge0.
  \end{equation*}
\end{theorem}

Short and elementary proofs of Harris's theorem are available \citep{Hairer2011, canizo_harris-type_2021}. The theorem just makes rigorous the fact that a uniform positivity condition plus a confinement condition imply exponential convergence to equilibrium. The central part of any arguments using Harris's theorem must then be contained in the proof of these conditions, and usually (like in our case) in the proof of the positivity condition.

\medskip Theorem \ref{spectral gap} automatically follows once we
prove the two conditions of Harris's Theorem with weight
$V(x) = \ap{x}^k$. We notice that, since solutions to (\ref{NLFP})
with $L^1_k$ initial data remain in $L^1_k$ for all times (and the
$\M_k$ norm in $L^1_k$ is just equal to the $L^1_k$ norm), Theorem
\ref{spectral gap} in $L^1_k$ really follows once we are able to apply
Theorem \ref{sHarris} in $\M_k$. A similar consideration holds with
exponential weights. Hence, the next two sections are devoted to
proving the hypotheses required in Theorem \ref{sHarris} for the
semigroup generated by equation (\ref{NLFP}).

\subsection{Lyapunov condition}

We define $L_\e^*$, the formal dual operator of $L_\e$, acting on
functions $\f\in \mathcal{C}^2$ (with appropriate growth as $|x| \to +\infty$)
as
\begin{equation*}\label{eq: dualL}
  L_\e^*\f:=A^*_\e\f+B_\e^*\f=\frac{1}{\e^2}\ird J_\e(y)(\f(x+y)-\f(x))\d y-x\nabla \f.
\end{equation*}
This ``appropriate growth'' is just to ensure the convolution
$J_\e * \varphi$ is properly defined. Let $(S_t)_{t\ge0}$ be the
semigroup associated to equation \eqref{NLFP}.

\begin{theorem}\label{thm:lyapunovpow}
  Let $J$ satisfy \eqref{eq:J-moments-basic} and assume $J\in L^1_{k'}$
  for some $k'\ge 2$. Then for every $0 < k \le k'$ there exist constants
  $C, \lambda>0$ depending only on $k$ and $J$ such that
  $$L_\e^*(\ap{x}^k)\le C-\lambda \ap{x}^k$$
  for any $0 < \e \leq 1$.
\end{theorem}

\begin{proof}
  To lighten the notation in the following, we will use $C_1, C_2$ as
  general constants, whose value may change from line to line.  Recall
  from \eqref{eq:J-moments-basic}, that for all $i=1,\dots,d$
	\begin{equation*}
		\int J(z) z_i\d z=0.
	\end{equation*}
	Therefore, for every $\f\in \mathcal{C}^{\infty}$, expanding around $x$,
	\begin{equation*}\begin{split}
			\frac{1}{\e^2}\ird J_\e(y)(\f(x+y)-\f(x))\d y&=\frac{1}{\e^2}\ird J(z)(\f(x+\e z)-\f(x))\d z\\
			&=\frac{1}{\e^2}\ird J(z)\Big(\f (x)+\frac{1}{2}\sum_{|\alpha|=2}R_\alpha(\xi)\e^2 z^{\alpha}-\f(x)\Big)\d z\\ 
                &=\frac{1}{2}\ird J(z) \sum_{i=1}^d \partial_i^2\f(\xi)z_i^{2}\d z
		\end{split}
              \end{equation*}
        where $R_\alpha(\xi)$ is the Lagrange reminder for some
        $\xi(x,z)\in$ $[x,x+\e z]$, the segment that joins $x$ and
        $x+\e z$.

  Fix now $\f(x)=\ap{x}^k$. With a simple computation 
  \begin{equation}\label{eq: 2deriv}
    \partial_i^2 \ap{x}^k = k\big(\ap{x}^{k-2}+(k-2)x_i^2\ap{x}^{k-4}\big).
  \end{equation}
  For $k\in(0,2)$ we have $\partial_i^2\f(\xi)\le k$, for every
  $i,\xi$ . Hence
  $$
  \frac{1}{\e^2}\ird J_\e(y)(\ap{x+y}^k-\ap{x}^k)\d y\le\frac{k}{2}\ird J(z) |z|^2\d z=k.
  $$
  For $k\ge 2$, since $\ap{\cdot}^j$ is a radially nondecreasing
  \textit{subadditive} function for every $j\ge 0$, we can bound
  \eqref{eq: 2deriv} by
 \begin{equation*}
 \begin{split}
     \partial_i^2\f(\xi)&\le k(k-1)\langle \xi\rangle^{k-2}
     \le k(k-1)\ap{|x|+\e |z|}_{k-2}\le C_1(\ap{x}^{k-2}+\ap{z}^{k-2}).
     \end{split}
 \end{equation*}
 It follows that 
 \begin{equation*}
     \begin{split}
        \frac{1}{\e^2}&\ird J_\e(y)(\ap{x+y}^k-\ap{x}^k)\d y
         \le\frac{1}{2}\ird J(z) \sum_{i=1}^d z_i^2 C_1(\langle x\rangle^{k-2}+\langle z \rangle^{k-2}) \d z\\
         &\le C_1\langle x\rangle^{k-2}\frac{1}{2}\ird J(z) |z|^2\d z+C_1\ird  J(z) \langle z \rangle^{k}\d z\le  C_2 \langle x\rangle^{k-2}.
     \end{split}
 \end{equation*}
Consequently, for every $k \ge 0$
	\begin{equation}\label{eq: Lyappowfinal}
		\begin{split}
			L^*_\e&\ap{x}^k=\frac{1}{\e^2}\ird J_\e(y)(\ap{x+y}^k-\ap{x}^k)\d y-x\nabla \f \le  C_2\ap{x}^{k-2}-k|x|^2 \ap{x}^{k-2}  \\
   &=  (C_2+k)\ap{x}^{k-2}-k(1+|x|^2 )\ap{x}^{k-2}
   \le C-\lambda 
			\ap{x}^k
		\end{split}
	\end{equation}
	for some $C,\lambda>0$.
 \end{proof}
\begin{remark}
  We point out that as $k$ approaches $0$, the right hand side of
  \eqref{eq: Lyappowfinal} goes to $0$ as well, since we don't have
  Lyapunov confinement in $L^1$ without weight.
\end{remark}

With almost the same procedure, we can prove a Lyapunov condition in
$L^1_{\exp,a}$.

\begin{theorem}
  \label{thm:lyapunovexp}
  Assume $J\in L^1_{\exp, a'}$ for some $a'>0$.  Then for every
  $0 < a < a'$ there exist $C, \lambda>0$ (only depending on $a$, $d$
  and $J$), such that
 	$$L_\e^*(e^{a\ap{x}})\le C-\lambda e^{a\ap{x}}.$$ 
\end{theorem}

\begin{proof}
  To lighten the notation in the following, we will use $C_1, C_2$ as
  general constants, whose value may change from line to line.  As before, for all $i=1,\dots,d$
	\begin{equation*}
		\ird  J(z) z_i\d z=0.
	\end{equation*}
	As before, for every $\f\in \mathcal{C}^{\infty}$, we Taylor expand around $x$, obtaining
	\begin{equation*}\begin{split}
			\frac{1}{\e^2}\ird J_\e(y)(\f(x+y)-\f(x))\d y=\frac{1}{2\e^2}\ird J(z) \sum_{i=1}^d \partial_i^2 \f(\xi)\e^2 z_i^{2}\d z
		\end{split}
	\end{equation*}
 where the Lagrange reminder is evaluated at some  $\xi(x,z)\in$ $[x,x+\e z]$.
 
 Set now $\f=e^{a\ap{x}}$. One can see that there exists $C_1>0$ a
 dimension dependent constant, such that
 $$\partial_i^2\f(\xi)=\frac{ae^{a\ap{\xi}}}{\ap{\xi}^2} \Big(
 a\xi_i^2+\ap{\xi} - \frac{\xi_i^2}{\ap{\xi}} \Big) \le
 C_1e^{a\ap{\xi}} \le C_1e^{a(\ap{x}+\e\ap{z})}$$ since $\f$ is a
 symmetric, positive and increasing function in every component and
 $\ap{\xi} \leq \langle |x|+\e |y| \rangle\le\ap{x}+\e\langle y
 \rangle$. Thus
 \begin{equation*}
     \begin{split}
         \frac{1}{\e^2}\ird J_\e(y)(\f(x+y)-\f(x))\d y
         &\le \frac{1}{2}\ird J(z)\sum_{i=1}^dC_1e^{\langle |x|+\e|z| \rangle}z_i^2 \d z
         \\
         &
         \le \frac{C_1}{2}e^{a\ap{x}}
         \ird J(z)e^{a\e\ap{z}}|z|^2 \d z
         \le  C_2e^{a\ap{x}}
     \end{split}
 \end{equation*}
since $$\ird J(z)|z|^2e^{a\e\ap{z}}\d z \le C_3\ird J(z)e^{a'\ap{z}}\d z<\infty.$$
(Let us point out that this is where the remark at the end of Theorem \ref{spectral gap} comes from). Then
	\begin{equation*}
		\begin{split}
			L^*_\e e^{a\ap{x}}=\frac{1}{\e^2}\ird J_\e(y)(\f(x+y)-\f(x))\d y-x\nabla \f & \le   C_2e^{a\ap{x}}-a\frac{|x|^2}{\ap{x}}e^{a\ap{x}}\le C-\lambda 
			e^{a\ap{x}}
		\end{split}
	\end{equation*}
	for some $C,\lambda>0$.
\end{proof}
We thus have proved a Lyapunov condition on the generator for
$V=\ap{\cdot}^k$ and $V=e^{a\ap{\cdot}}$. For these weight functions
(and in general for $\mathcal{C}^2$ functions), we see that actually
Assumption \eqref{eq:sLyap} implies Assumption \eqref{eq:
  semigroupLyap}, which is the statement of the next Theorem.

\begin{theorem}[Semigroup Lyapunov condition]\label{thm: finalLyap}
Assume $u_0$ and $J$
  satisfy the hypothesis of Theorem \ref{thm:lyapunovpow}
  (respectively \ref{thm:lyapunovexp}). For $\f=\ap{x}^k$
  (respectively $\f= e^{a\ap{x}}$), the semigroup associated to
  equation \eqref{NLFP} satisfies a Lyapunov condition 
  \eqref{eq: semigroupLyap} with $V\equiv \f$ and the same $C$ and $\lambda$ of Theorem
  \ref{thm:lyapunovpow} (respectively \ref{thm:lyapunovexp}).
    
\end{theorem}
\begin{proof}
We just prove it for $L^1_k$, being analogous in the other case. 

    Let  $u_0$ be in $\mathcal{D}(L_\e)$ by standard semigroup theory
$ u(t)\in \mathcal{C}^1([0,\infty);L^1_k)$, and consequently we compute
\begin{equation*}
\begin{split}
    \ddt \ird u(t,\cdot) \f(x)\d x&= \ird \ddt u(t,\cdot) \f(x) \d x  \\
    &=\ird (Lu) \f (x)\myeq \ird u(L^* \f)\le  \ird u(-\lambda \f+ C)
    \end{split}
\end{equation*}
where to prove $\myeq$ we use classic cutoff arguments with bump functions. 

We then use Gronwall Lemma with the $\mathcal{C}^1$ function $f(t):= \ird u(t)\f$:
That is,$$
f'(t)\le -\lambda f(t)+ C\implies f(t)\le e^{-\lambda t}f(0)+ \frac{C}{\lambda}(1-e^{-\lambda t}).$$

Since the $\mathcal{D}(L_\e)$ is dense in $L^1_k$, the claim is true also for  any $u_0\in L^1_k$ by classic density argument.
\end{proof}

\subsection{An $L^\infty$ version of Berry-Esseen's central limit theorem}
\label{sec:BerryEsseen}

Berry-Esseen-type theorems are quantitative versions of the central
limit theorem which provide an estimate on the error between the
normal distribution and the scaled law. A result of this kind will be
crucial for our proof in Section \ref{sec:posit} of positivity for the
nonlocal Fokker-Planck equation. The first version of this theorem was
proved independently by \cite{berry_1941} and
\cite{esseen1942liapunov} for real, centered i.i.d.~random variables
$X_1,\dots X_n$ with $\Var(X_1)=\sigma^2$ and $\mathbb{E}(|X_1|^3)=\rho$,
obtaining a result of the type

$$
\sup_{x\in\R}|F_n(x)-\Phi(x)|\le \frac{C\rho}{\sqrt{n} \sigma^3}
$$
where $\Phi$ is the cumulative distribution function (c.d.f.) of the
normal distribution, $F_n$ is the c.d.f.~of
$\bar S_n=\frac{1}{\sigma \sqrt{n}}\sum_{i=1}^nX_i$, and $C$ is a
universal constant depending on the dimension $d$. Several
generalizations have been developed, proving the same kind of results
for not i.i.d.~random vectors in $\R^d$, as well as improving the
precision of the value of $C$. Recent versions can be found in
\cite{bentkus_2005} and \cite{raic_2019}.

The majority of this kind of results are ``weak'' results in the sense
that they deal with the cumulative distribution function. For our results we
would like to work with the density instead. An important result in
this direction is \cite{lions1995strengthened} in which the authors
provide a stronger version of the CLT on the densities, and later in
\cite{goudon_junca_toscani_BE_2002}, {\cite{carlen_entropy_2010},
  \cite{mischler_kacs_2014} }. For this paper we need an estimate on
the densities of independent but not identically distributed random
vectors in $\R^d$, so we will extend the proofs of the previous cited
papers to our case. Following the ideas in
\cite{goudon_junca_toscani_BE_2002}, we weaken the hypothesis of
finiteness of the third moment. We require only a ``bit more''
than finite variance, paying a price on the speed of convergence
to zero of the error.

Let $f:\R^d\to[0,\infty)$ be a probability density function such that
there exists $s\in(0,1]$ with
\begin{equation}\label{assBE}
  \int_{\R^d}xf(x)\d x=0,
  \quad
  \int_{\R^d} x_ix_jf(x)\d x=\delta_{ij},
  \quad
  \int_{\R^d} |x|^{2+s}f(x)\d x:=\rho_{2+s}<\infty.
\end{equation}
Notice that this normalisation of moments of order $2$ is different
from \eqref{eq:J-moments-basic} by a factor of $2$. We prefer the
present normalisation in this section, which is a natural one in a
probabilistic setting and is the same one used in most statements of
Berry-Esseen's theorem. By a simple scaling, it is straightforward to
apply the result if one assumes the normalisation
(\ref{eq:J-moments-basic}) instead.

The main result of this section is the following $L^\infty$ version of
Berry-Esseen's Theorem for independent, nonidentically
distributed random vectors:

\begin{theorem}\label{infBE}
  Let $p\in (1,\infty)$ and $f\in L^1(\R^d)\cap L^{p}(\R^d)$ satisfying
  \eqref{assBE}. Let $(\sigma_i)_{i \geq 1}$ be a sequence of positive
  numbers, and define for $n \geq 1$,
  \begin{gather*}
    f^{\sigma_1,\dots\sigma_n}(x):=f_{\sigma_1}*f_{\sigma_2}*\cdots*f_{\sigma_n}(x)
    \\
    \mathfrak{f}_n(x):=(\bar{\sigma}^2 n)^{d/2}f^{\sigma_1,\dots\sigma_n}(\sqrt{n}\bar \sigma x),
  \end{gather*}
  with $\bar\sigma^2=\frac{1}{n}\sum_{i=1}^n\sigma_i^2$.
  Suppose, in addition, that there exist $l,L>0$ such that
  $$0<l\le \sigma_i \le L\qquad i=1,\dots, n$$      
  Then, there exist an integer $N=N(p)$ and a constant $C_{BE} $
  depending on  $\frac{L}{l},p, d,\rho_{2+s},\norm{f}_{L^p}$ such that $\Jnorm
  \in L^\infty$ for all $n \geq N$ and
\begin{equation*}
    \norm{\Jnorm-G}_{L^\infty}\le \frac{C_{BE}}{n^{s/2}}\quad \text{ for every $n\ge N$,}
\end{equation*}
where $G$ is the density of the standard normalized Gaussian.
\end{theorem}

In order to prove this we first present some preliminary results. The
first one is a generalization of the identity
$a^n-b^n=\sum_{i=1}^n (a-b)a^{i-1}b^{n-i}$, that can be easily proven
by induction:
\begin{lemma}\label{produttoria}
    For real $a_1,\dots,a_n$ and $b_1,\dots b_n$, for any $n\ge1$
    \begin{equation*}
        \prod_{i=1}^na_i-\prod_{i=1}^nb_i=\sum_{i=1}^n\Big[(a_i-b_i)\big(\prod_{j=1}^{i-1}b_j\big)\big(\prod_{j=i+1}^n a_j\big)\Big]
    \end{equation*}
\end{lemma}
The second one is a H\"older regularity result for the Fourier transform.

\begin{lemma}\label{lem: holdfou}
  Let $f:\R^d\to \R$ be a probability distribution such that for some
  $s \in (0,1)$,
  $$\ird f(x)|x|^{2+s}\d x=\rho_{2+s} < +\infty.$$
  Then $\hat{f}\in \mathcal{C}^{2,s}$; that is, $\hat{f}\in \mathcal{C}^2$ and for every
  $x,y\in \R^d$
  $$\max_{|\alpha|=2}|D^\alpha \hat{f}(x)-D^\alpha \hat{f}(y)|\le
  4 \rho_{2+s}|x-y|^{s}. $$ 
\end{lemma}

\begin{proof}
  The $\mathcal{C}^2-$regularity is immediate from Fourier transform
  properties. We are left to show the H\"older-continuity of the
  derivatives of order $2$, namely, for every multi-index $\alpha$
  such that $|\alpha|=2$,
  \begin{equation*}
    \begin{split}
      &\abs{D^{\alpha}\hat{f}(\xi)
      -D^{\alpha}\hat{f}(\eta)}=\abs{-\widehat{x^\alpha
        f}(\xi)+\widehat{x^\alpha f}(\eta)}=\abs{\ird
        \Big(e^{-ix\eta}-e^{-ix\xi})x^\alpha f(x)\d x}
      \\
      &\le \int_{|x|\le R}|e^{-ix\xi}-e^{-ix\eta}||x|^2|f(x)\d
        x+\int_{|x|>R}\abs{e^{-ix\xi}-e^{-ix\eta}}|x|^2|f(x)\d x
      \\
      &\le  2|\xi-\eta|\int_{|x|\le R}|x|^{2+s}|x|^{1-s}f(x)\d
        x+2R^{-s} \int_{|x|>R}|x|^{2+s}f(x)\d x
      \\
      &\leq
        2\rho_{2+s}(|\xi-\eta|R^{1-s}+R^{-s}).
    \end{split}
  \end{equation*}
  Choosing $R=|\xi-\eta|^{-1}$ we conclude the proof.
\end{proof}

The last ingredient we need is the following technical Lemma (that is
basically \cite[lemma 4.8]{mischler_kacs_2014} with a weaker
assumption on moments. We denote by $B_\delta(x_0)$, the ball in $\R^d$
centered at $x_0$ with radius $\delta$. When the ball is centered at $0$,
we use $B_\delta$ as a shorthand.
\begin{lemma}\label{lemfou}
  Let $f$ be a probability distribution on $\R^d$ satisfying
  \eqref{assBE}. Then
  \begin{enumerate}[(i)]
  \item  There exist $\delta\in(0,1)$ depending on $\rho_{2+s}$ such that 
    \begin{equation*}
      \forall \xi\in B_\delta \quad |\hat{f}(\xi)|\le e^{-\frac{|\xi|^2}{4}}
    \end{equation*}
  \item Assume additionally that $f\in L^p$, with
    $p\in(1,\infty]$. For any $\delta>0$ there exists
    $\kappa=\kappa(\rho_{2+s},\delta, \norm{f}_{L^p})\in(0,1)$ such
    that
    $$\sup_{\abs{\xi}\ge\delta}|\hat{f}(\xi)|\le \kappa(\delta).$$
\end{enumerate}     
\end{lemma}
\begin{proof}
\begin{enumerate}[(i)]
\item For $\hat{f}:\R^d\to \R$ and for $\xi$ define $h:[0,|\xi|]\to \R$ such that $h(t):=\hat{f}\Big(t\frac{\xi}{|\xi|}\Big)$.
Applying twice the fundamental theorem of calculus to  $h$ we write
\begin{equation*}
    \hat{f}(\xi)=h(|\xi|)=1-\frac{|\xi|^2}{2}+ \int_0^{|\xi|} (h''(z)-h''(0))(z-|\xi|)\d z
\end{equation*}
where 
\begin{equation*}
\begin{split}
h''(z)=\Big(\frac{\xi}{|\xi|}\Big)^TH_{\hat{f}}\Big(z\frac{\xi}{|\xi|}\Big)\Big(\frac{\xi}{|\xi|}\Big).
\end{split}
\end{equation*}
Hence, applying Lemma \ref{lem: holdfou}, we can bound the reminder 
\begin{equation*}
    \begin{split}
        &\abs{\int_0^{|\xi|}\Big(\frac{\xi}{|\xi|}\Big)^T\Big(H_{\hat{f}}\Big(z\frac{\xi}{|\xi|}\Big)-H_{\hat{f}}(0)\Big)\Big(\frac{\xi}{|\xi|}\Big)(z-|\xi|)\d z}\\
        &\le \int_0^{|\xi|} \tnorm{(H_{\hat{f}}\Big(z\frac{\xi}{|\xi|}\Big)-H_{\hat{f}}(0)}|z-|\xi||\d z\\
        &\le  4 \rho_{2+s} \int_0^{|\xi|} |z|^s|z-|\xi||\d z=\frac{ 4 \rho_{2+s}|\xi|^{2+s}}{(s+1)(s+2)},
    \end{split}
\end{equation*}
where $H_f(x)$ is the Hessian matrix of $f$ applied at point $x$ and $\tnorm{\cdot}$ is the spectral norm which is always less or equal than the $\max$- norm used in Lemma \ref{lem: holdfou}.

Thus, for small $|\xi|<\delta$ which explicitly depends only on $s$ and $\rho_{2+s}$, we have 
$$
\abs{\hat{f}(\xi)}\le \abs{1-\frac{3}{8}|\xi|^2}\le e^{-\frac{|\xi|^2}{4}}.
$$
\item This statement is a direct consequence of \cite[Prop. 26
  \textit{(iii)}]{carlen_entropy_2010}, where the assumption of
  finiteness of the entropy is ensured by our assumption that
  $f \in L^1\cap L^p$ and has some finite moment.
  \qedhere
\end{enumerate}

\end{proof}
We can now present the proof of our version of Berry-Esseen theorem,
that follows the ideas presented in
\cite[Thm. 1]{goudon_junca_toscani_BE_2002} and \cite[Thm.
4.6]{mischler_kacs_2014}.

\begin{proof}[Proof of Theorem \ref{infBE}]
	First, we notice that $\mathfrak{f}_n$ satisfies the hypotheses of Lemma \ref{lemfou}  for every $n$, and, in particular
	$$\ird\mathfrak{f}_n(x)x_ix_j\d x=\delta_{ij}.$$
	We can rewrite
	\begin{equation*}
		\begin{split}
			{\Jnormh}(\xi)&=\ird e^{-ix\xi} (\bar \sigma^2 n)^{d/2}f^{\sigma_1\dots,\sigma_n}(\sqrt{n}\bar \sigma x)\d x=\ird e^{-iy \frac{\xi}{\sqrt{n}\bar \sigma}}f_{\sigma_1}*\cdots*f_{\sigma_n}(y)\d y\\
			&=\mathcal{F}(f_{\sigma_1}*\cdots *f_{\sigma_n})\Big(\frac{\xi}{\sqrt{n}\bar \sigma}\Big)=\prod_{j=1}^n\hat{f}_{\sigma_j}\Big(\frac{\xi}{\sqrt{n}\bar \sigma}\Big)=\prod_{j=1}^n \ird e^{-ix \frac{\xi}{\sqrt{n}\bar \sigma}}\sigma_j^{-d}f\Big(\frac{x}{\sigma_j}\Big) \d x \\
			&=\prod_{j=1}^n \ird e^{-iy\frac{\xi \sigma_j}{\sqrt{n}\bar \sigma}}f(y)\d y=\prod_{j=1}^n\JJi,
		\end{split}
	\end{equation*}
	and, since $G$ is the standard gaussian, its Fourier transform satisfies
	\begin{equation*}
		\hat{G}(\xi)=\prod_{j=1}^n \GGi.
	\end{equation*}
	Without loss of generality, we can assume $p\le2$. Hausdorff-Young's
	inequality then implies that $\hat{f}$ is in $L^{p'}\cap L^\infty$
	with $p'\in [2,\infty)$. Using H\"older's inequality,
	$$\begin{aligned}
\norm{\prod_{j=1}^n\JJi}_{L^1}&\le \prod_{j=1}^n\norm{\JJi}_{L^n}=
	\prod_{j=1}^n\Big(\frac{\bar\sigma
		\sqrt{n}}{\sigma_j}\Big)^{d/n}\norm{\hat{f}}_{L^n}\\&\le
	n^{d/2}\Big(\frac{L}{l}\Big)^d\norm{\hat{f}}_{L^n}^n,
\end{aligned}$$
	which is
	finite for $n\ge p'$, i.e.  $\Jnormh$ is in $L^1$ for any $n\ge
	p'$. We can thus control the $L^\infty$ norm with $L^1$ norm of the
	Fourier transform, since
	\begin{equation*}
		\abs{\Jnorm(x)-G(x)}
		=
		(2\pi)^{-d}\abs{\ird (\Jnormh (\xi)-\hat G(\xi))e^{i\xi x}\d \xi}
		\le
		(2\pi)^{-d}\ird \abs{\Jnormh-\hat G} \d \xi
	\end{equation*}
	Let $\delta$ be given by Lemma \ref{lemfou} $(i)$ for $\hat f$ and
	define
	$B_\delta=B_{\delta}(n,l,L):=\{ \xi: \abs{\xi}<\sqrt{n}
	\frac{l}{L}\delta \}$.  Following known techniques, we split the
	integral into high and low frequencies
	\begin{equation*}
		\norm{\Jnorm-G}_{L^\infty}\le \int_{\xi\in B^c_{\delta}}|\Jnormh|\d \xi+\int_{\xi\in B^c_{\delta}}|\hat G|\d \xi+\int_{\xi\in B_{\delta}}|\Jnormh-\hat G|\d \xi:=T_1+T_2+T_3
	\end{equation*}
	The first two terms can be bounded in a similar way. Without loss of generality, assume $p'$ is an integer (otherwise, we
	can work with $\cl{p'}$, the smallest integer larger or equal than
	$p'$, since $\hat{f}$ is also in $L^{\cl{p'}}$). Then
	\begin{multline}
		\label{dod}
		\int_{B_\delta^c} \prod_{j=1}^n\abs{\JJi} \d \xi
		=
		n^{d/2}\int_{|\eta| \ge\frac{l}{L}\delta}
		\prod_{j=1}^n \abs{\hat f \Big(\frac{\sigma_j}{\bar \sigma} \eta\Big)} \d\eta
		\\
		\le
		n^{d/2}  \sup_{|\eta| \ge\frac{l}{L}\delta} \abs{\hat f \Big(\frac{\sigma_{p'+1}}{\bar \sigma} \eta\Big)}\cdots \sup_{|\eta| \ge\frac{l}{L}\delta} \abs{\hat f\Big(\frac{\sigma_n}{\bar \sigma} \eta\Big)}\\ \int_{|\eta| \ge\frac{l}{L}\delta} \abs{\hat f \Big(\frac{\sigma_1}{\bar \sigma} \eta\Big)}\cdots
		\abs{\hat f\Big(\frac{\sigma_{p'}}{\bar \sigma} \eta\Big)}
		\d\eta.
	\end{multline}
	For each $j=p'+1,\dots,n$
	\begin{equation}\label{eq: supdelta}
		\sup_{|\eta| \ge\frac{l}{L}\delta}
		\abs{\hat f \Big(\frac{\sigma_{j}}{\bar \sigma} \eta\Big)}
		=
		\sup_{|\eta'| \frac{\bar\sigma}{\sigma_j}
                  \geq
                  \frac{l}{L} \delta}
		| \hat f ( \eta' ) |
                \leq
                \sup_{|\eta'| \geq \frac{l^2}{L^2} \delta}
		| \hat f ( \eta' ) |
                \leq \kappa(\delta)
	\end{equation}
	with $\kappa(\delta)\in (0,1)$, coming from Lemma \ref{lemfou} $(ii)$.
	The factors inside the integral in Equation \eqref{dod} can be bounded
	by using H\"older's inequality:
	we rewrite
	\begin{equation}\label{eq: BEholdineq}
		\begin{split}
			\int_{|\eta| \ge\frac{l}{L}\delta}
			&\abs{\hat f \Big(\frac{\sigma_1}{\bar \sigma} \eta\Big)}\cdots
			\abs{\hat f\Big(\frac{\sigma_{p'}}{\bar \sigma} \eta\Big)}
			\d\eta
			\\
			&\le \bigg(\ird \abs{\hat f \Big(\frac{\sigma_1}{\bar \sigma} \eta
				\Big)\d\eta}^{p'}\bigg)^{1/p'}\cdots  \bigg(\ird \abs{\hat f
				\Big(\frac{\sigma_{p'}}{\bar \sigma} \eta\Big)
				\d\eta}^{p'}\bigg)^{1/p'}
			\\
			&=\Big(\frac{\bar \sigma}{\sigma_1}\Big)^{d/p'}
			\|\hat f\|_{p'}\cdots\Big(\frac{\bar \sigma}{\sigma_{p'}}\Big)^{d/p'}
			\|\hat f\|_{p'}\le \Big(\frac{L}{l}\Big)^{d}
			\|\hat{f}\|_{p'}^{p'}\le\Big(\frac{L}{l}\Big)^{d}C_p \|f\|_{p}^{p'},
		\end{split}
	\end{equation}
	after first applying H\"older's inequality
	and then Hausdorff-Young's inequality ($C_p$ is the Hausdorff-Young constant).
	Hence, plugging \eqref{eq: supdelta} and \eqref{eq: BEholdineq} into equation
	\eqref{dod} we obtain
	\begin{equation}
		T_1\le C_p\Big( \sqrt{n}\frac{L}{l}\Big)^{d} \kappa(\delta)^{n-p'}\|f\|_p^{p'}.
	\end{equation}
	We proceed analogously for $T_2$, obtaining a very similar bound.
	Since $\kappa(\delta)<1$, $\kappa(\delta)^{n-p'}$ decays very fast in $n$, thus for $n$ large enough (depending on $p$), there exists $C_1>0$ such that
	\begin{equation}\label{eq: 21bis}
		T_1+T_2\le \frac{C_1}{n^{s/2}}
	\end{equation}
	Now we focus our attention on $T_3$. Using the factorization in Lemma
	\ref{produttoria} we obtain
	\begin{multline*}
		\label{eqA}
		\abs{\Jnormh(\xi)-\hat G(\xi)}=\prod_{j=1}^n \JJi-\prod_{i=1}^n \GGi\\
		=\sum_{j=1}^n \abs{\JJi-\GGi}\abs{\prod_{k=1}^{j-1}\JJj}\abs{\prod_{k=j+1}^{n}\GGj}.
	\end{multline*}
	From Lemma \ref{lemfou} $(i)$ it follows that for $\xi\in B_\delta$
	\begin{equation*}
		\abs{\hat{f}\Big(\frac{\xi \sigma_k}{\sqrt{n}\bar\sigma }\Big)}\le e^{-\frac{|\xi|^2}{4n}\frac{\sigma_k^2}{\bar \sigma^2}}\le e^{-\frac{|\xi|^2}{4n}(\frac{l}{L})^2}.
	\end{equation*}
	and
	\begin{equation*}
		\abs {\hat{G}\Big({\frac{\xi \sigma_k}{\sqrt{n}\bar \sigma}}\Big)}=e^{-\frac{|\xi|^2}{2n}\frac{\sigma_k^2}{\bar \sigma^2}}\le e^{-\frac{|\xi|^2}{2n}(\frac{l}{L})^2  }.
	\end{equation*}
	Then
	\begin{equation*}
		\begin{split}
			&\frac{\abs{\Jnormh(\xi)-\hat G(\xi)}}{|\xi|^{2+s}}
			\\&=\frac{1}{\abs{\xi}^{2+s}}\Big[ \sum_{j=1}^n
			\abs{\JJi-\GGi}\abs{\prod_{k=1}^{j-1}\JJj}\abs{\prod_{k=j+1}^{n}\GGj}\Big]
			\\
			&\le \frac{1}{\abs{\xi}^{2+s}} \sum_{j=1}^n \abs{\JJi-\GGi}
			e^{-\frac{|\xi|^2}{4n}(j-1)(\frac{l}{L})^2}e^{-\frac{|\xi|^2}{2n}(n-j)(\frac{l}{L})^2
			}
			\\
			&\le \frac{1}{\abs{\xi}^{2+s}} \sum_{j=1}^n \abs{\frac{\xi
					\sigma_j}{\sqrt{n}\bar \sigma}}^{2+s} \frac{\abs{\JJi-\GGi}}{
				\abs{\frac{\xi \sigma_j}{\sqrt{n}\bar
						\sigma}}^{2+s}}e^{-\frac{|\xi|^2}{4}\frac{n-1}{n}(\frac{l}{L})^2}
			\\
			&\le \frac {1}{n^{\frac{2+s}{2}}}\Big(\frac{L}{l}\Big)^{2+s}n
			\sup_{\eta'}\frac{|\hat{f}(\eta')-\hat
				G(\eta')|}{\abs{\eta'}^{2+s}}
			e^{-\frac{\abs{\xi}^2}{8}(\frac{l}{L})^2 }.
		\end{split}
	\end{equation*}
	Since $f $ and $G$ are both centered and have the same covariance
	matrix, in the same spirit of the proof of Lemma \ref{lemfou} $(i)$,
	we obtain
	\begin{equation*}
\begin{split}
		\sup_{\eta'}\frac{|\hat{f}(\eta')-\hat G(\eta')|}{\abs{\eta'}^{2+s}}
		&\le
		\sup_{\eta'}\left(
		\frac{|\hat{f}(\eta')-\hat{f}(0)|}{\abs{\eta'}^{2+s}}
		+\frac{|\hat G(\eta')-\hat G(0)|}{\abs{\eta'}^{2+s}}
		\right)
		\\&\le
		C_1(\rho_{2+s} + M_{2+s}(G)).
\end{split}
	\end{equation*}
	where $M_{2+s}(G)=\int |x|^{2+s} G(x)\d x$.
	This leads to the final estimate on the third term (with the
	``correct'' rate of decay), that is
	\begin{equation*}
		\begin{split}
			T_3&\le \frac{C_1}{n^{s/2}}\Big(\frac{L}{l}\Big)^{2+s}(\rho_{2+s}+M_{2+s}(G)) \ird e^{-\frac{\abs{\xi}^2}{8}(\frac{l}{L})^2 }|\xi|^{2+s}\d \xi\\
			&:=\frac{1}{n^{s/2}}\Big(\frac{L}{l}\Big)^{2+s}(\rho_{2+s}+M_{2+s}(G))C_{\frac{L}{l},d}.
		\end{split}
	\end{equation*}
	Taking into account \eqref{eq: 21bis} and adding  $T_3$, the proof is completed.
\end{proof}

\subsection{Positivity of the solution}
\label{sec:posit}

In order to show Harris's positivity condition (Hypothesis
\eqref{eq:HLB}) we are in fact able to prove a stronger result, namely
a local positivity condition at any positive time $t>0$:

\begin{theorem}[Lower bound]
  \label{posit}
  Let $J \: \R^d \to [0,\infty)$ satisfy 
  \eqref{eq:J-moments-basic}, \eqref{eq:smoment}, and   $J\in L^p$ for some $p\in (1,\infty].$ Then, the solution $u$ of equation
    \eqref{NLFP} with a probability initial condition $u_0$ is uniformly
    positive on compact sets for all $0 < \e \le 1$. That is:
    for every $0 < \e \leq 1$, $t>0$, $R_1 R_2 > 0$,
    there exists an $\e-$independent constant $\alpha,$
    such that the solution of \eqref{NLFP} with initial data $u_0\in L^1_k$ is such that 
  
  $$
   u(t,x) \ge \alpha \int_{B_{R_2}}u_0 \qquad \text{for } x\in B_{R_1}.
  $$
\end{theorem}

This is a nontrivial estimate since we seek the independence of
$\alpha$ from $\e$; in particular, we seek positivity that holds
uniformly in the limit $\e\rightarrow 0$.

Our proof of this estimate is based on the Wild sum representation of
the solution from Section \ref{sec:representation} and on the version
of Berry-Esseen's theorem we proved in Section
\ref{sec:BerryEsseen}. We distinguish two cases which are described by
the two following Lemmas:

\begin{lemma}\label{smalle2} Let $J$ satisfy \eqref{eq:J-moments-basic} and \eqref{eq:smoment}. Then, for any $t,\eta>0$, there exists  $\e_0>0 $ and a ($\e-$independent) constant
  $A>0$ such that
  $$\Jn(x)\ge A\qquad \text{ for all } x\in B_\eta $$
  for all $\e< \e_0$ and for   any  $t_1,\dots, t_n$ with $t \ge t_1\ge\dots t_n\ge 0$ and $n$ such that 
  $$\frac{t}{\e^2}\le n\le 2\frac{t}{\e^2}.$$

\end{lemma}

\begin{proof} 
  Define the independent random vectors $X_i\sim J_{\e e^{t_i}}$ for
  which
  $$\ird x J_{\e e^{t_i}}\d x=0\qquad \ird x_jx_k J_{\e e^{t_i}}(x)\d
  x=2\e^2e^{2t_i}\delta_{jk}.$$ This implies that
  $S_n=\sum_{i=1}^nX_i \sim \Jn $ and
$$
\ird x\Jn (x)\d x=0\qquad\ird x_jx_k\Jn(x)\d x=2\e^2 n \bar\tau^2\delta_{jk},
$$
where $\bar\tau^2:=\frac{1}{n}\sum_{i=1}^n e^{2t_i}$ and $1\le \bar\tau^2\le e^{2t}$.
Notice that $$l:=\e\le \e e^{t_n}\le\dots\le \e e^{t_1}\le \e e^t:=L$$ and thus $\frac{L}{l}=e^{t}.$
Let $N>0$ be the threshold for which the Berry-Esseen estimate
\ref{infBE} holds, and define $\e_1=\sqrt{\frac{t}{N}}$.
Let $\e_2>0$ be such that 
$$\Big[G\Big(\frac{\eta}{\sqrt{2t}}\Big)
-\frac{ C_{BE}}{t^{s/2} }\e_2^s\Big]\ge\frac{1}{2}G\Big(\frac{\eta}{\sqrt{2t}}\Big).$$
Take $\e_0:=\min\{\e_1,\e_2\}$. For all $\e<\e_0$, we can apply
Theorem \ref{infBE}, since, by hypothesis,
$n\ge \frac{t}{\e^2}\ge\frac{t}{\e_1^2}=N $.  Hence, since $G$ is
radially decreasing,
\begin{equation}
  \begin{split}
    \Jn (x)
    &\ge (2\e^2
      n\bar\tau^2)^{-d/2}\Big[G\Big(\frac{x}{\sqrt{2n}\e\bar\tau}\Big)-\frac{C_{BE}}{n^{s/2}}\Big]
    \\
    &\ge (2\e^2
      n\bar\tau^2)^{-d/2}\Big[G\Big(\frac{\eta}{\sqrt{2t} \bar
      \tau}\Big)-\frac{ C_{BE}}{t^{s/2}}\e^s\Big]
    \\
     &\ge (2\e^2
      n\bar\tau^2)^{-d/2}\Big[G\Big(\frac{\eta}{\sqrt{2t} \bar
      \tau}\Big)-\frac{ C_{BE}}{t^{s/2} }\e_2^s\Big]
    \\
    &\ge
    (4te^{2t})^{-d/2}\frac{1}{2}G\Big(\frac{\eta}{\sqrt{2t}}\Big):=A.
    \qedhere
  \end{split}
\end{equation}
%
\end{proof}
\color{black}
If $\e$ has a positive lower bound $\e_0$, then we prove the following
Lemma in the same spirit:
\begin{lemma}
  \label{Nlargee}
  Let $J$ satisfies \eqref{eq:J-moments-basic} and \eqref{eq:smoment}.
  Take $\e_0 \in (0,1]$, $t>0$ and $\eta>0$. Then, there exist an
  explicit $N\in \mathbb{N}, B>0$ such that for all $n\ge N$ and
  $\e \in [\e_0,1]$,
  $$\Jn(x)\ge\frac{B}{n^{\frac{d}{2}}}=:B_n\qquad \text{for all  }x\in B_\eta$$
  where $B$  depends on   $\e_0$, $t$, $d$, $\eta$ and the Berry-Esseen constant.
\end{lemma}

\begin{proof}
Let $N=\max\{N_1,N_2\}$,
where $N_1$ is the threshold for which estimate in Theorem \ref{infBE} holds, and $N_2$ is such that  $$G\Big(\frac{\eta}{\sqrt{2}\e_0}\Big)-\frac{C_{BE}}{N_2^{s/2}}\ge\frac{1}{2}G\Big(\frac{\eta}{\sqrt{2}\e_0}\Big).$$
Then for $n\ge N$, applying Berry - Esseen CLT, we obtain

\begin{equation*}
    \begin{split}
    \Jn (x)
    &\ge (2\e^2
      n\bar\tau^2)^{-d/2}\Big[G\Big(\frac{x}{\sqrt{2n}\e\bar\tau}\Big)-\frac{C_{BE}}{n^{s/2}}\Big]
    \\
    & \ge \frac{1}{n^{\frac{d}{2}}}\Big(\frac{1}{\sqrt{2}e^t}\Big)^d\Big[G\Big(\frac{\eta}{\sqrt{2}\e_0}\Big)-\frac{C_{BE}}{N_2^{s/2}}\Big]\\
    & \ge \frac{1}{n^{\frac{d}{2}}}\Big(\frac{1}{\sqrt{2}e^t}\Big)^d \frac{1}{2}G\Big(\frac{\eta}{\sqrt{2}\e_0}\Big)=:\frac{B}{n^{\frac{d}{2}}}.
    \end{split}
\end{equation*}   
\end{proof}
The last ingredient we need to prove Theorem \ref{posit} is
 the following lemma on the partial sums of the densities of Poisson-distributed random variables.
\begin{lemma}
  \label{2n}
  There exists an explicit constant $C_L$ such that
  $$s_m := e^{-m}\sum_{n=m}^{2m}\frac{m^n}{n!}\ge C_L
  \qquad \text{for all $m \geq 1$,}$$ and
  $\lim_{m\to\infty}s_m=\frac{1}{2}.$
\end{lemma}
\begin{proof}
  Let $X_m$ be a Poisson random variable of parameter $m$, so that
  \begin{equation*}
    s_m = \mathbb{P}(m \leq X_m \leq 2m).
  \end{equation*}
  For a fixed $m$, using that $\mathbb{P}(X_m \geq m) \geq 1/2$ and the
  Chernoff-type bound $\mathbb{P}(X_m \geq x) \leq e^{-m} (e m)^x /
  x^x$ (for $x > m$), we have that
  \begin{equation*}
    s_m = \mathbb{P}(X_m \geq m) - \mathbb{P}(
    X_m \geq 2m+1) \ge \frac{1}{2}- e^{m+1 }
    \Big(\frac{m}{2m+1}\Big)^{2m+1} := b_m \geq b_1 =: C_L,
  \end{equation*}
  since $b_m$ is increasing for $m\ge 1$. To prove that $s_m \to 1/2$,
  let $Y_i,$  $i=1,\dots, m$ be Poisson-distributed random variables with
  parameter $1$ so that $X_m= \sum_{i=1}^mY_i$. We have
  \begin{equation*}\begin{split}
      &e^{-m}\sum_{n=m}^{2m}\frac{m^n}{n!}=e^{-m}\sum_{n=0}^{2m}\frac{m^n}{n!}-e^{-m}\sum_{n=0}^{m-1}\frac{m^n}{n!}
			\\
			&= \mathbb{P}\Big(\sum_{i=1}^mY_i-m\le{m} \Big)-\mathbb{P}\Big(\frac{1}{\sqrt{m}} \big({\sum_{i=1}^mY_i-m}\big)\le -\frac{1}{\sqrt{m}}\Big)\ \rightarrow1-\frac{1}{2}= \frac{1}{2},
		\end{split}
	\end{equation*}
	thanks to the law of large numbers and the central limit theorem.
\end{proof}

\begin{remark}
  For our proofs we will use the above lemma as stated. However, with
  a very similar proof one can show that for any $r > 1$, the lemma
  still holds with a partial sum from $m$ to $\cl{rm}$ (instead of $m$
  to $2m$), with the same limit of $1/2$.
\end{remark}

Proving Theorem \ref{posit} is now straightforward:

\begin{proof}[Proof of Theorem \ref{posit}]
Expressing the solution via Wild sums we have
$$
u(t,x)\ge e^{(d-\frac{1}{\e^2})t} \sum_{n=1}^\infty \e^{-2n}\int_0^t\dots\int_0^{t_{n-1}}J_\e^{t_1,\dots,t_n}*u_0(e^tx) \d t_n\dots \d t_1.
$$
Set $\eta=e^tR_1+R_2$ and take $\e_0$ from Lemma \ref{smalle2}. If
$\e\le\tilde{\e}_0:= \min \{\e_0, \sqrt{t}\}$, applying Lemma
\ref{smalle2} we have
\begin{equation*}\label{eq: finalsmall}
\begin{split}
u(t,x)&\ge e^{(d-\frac{1}{\e^2})t} \sum_{n=1}^\infty \e^{-2n}\int_0^t\dots\int_0^{t_{n-1}}\ird J_\e^{t_1,\dots,t_n}(e^tx-y)u_0(y)\d y \d t_n\dots \d t_1\\
&\ge e^{(d-\frac{1}{\e^2})t} \sum_{n=\fl{\frac{t}{\e^2}}}^{2\fl{\frac{t}{\e^2}}}\e^{-2n}\int_0^t\dots\int_0^{t_{n-1}}\int_{B_{R_2}}J_\e^{t_1,\dots,t_n}(e^tx-y)u_0(y)\d y \d t_n\dots \d t_1\\
&\ge A e^{d t}e^{-\frac{t}{\e^2}} \sum_{n=\fl{\frac{t}{\e^2}}}^{2\fl{\frac{t}{\e^2}}}
\Big(\frac{t}{\e^2}\Big)^n\frac{1}{n!} \int_{B_{R_2}}u_0(y)\d y\\
&\ge \Big(A e^{\d t}e^{-(\frac{t}{\e^2}-\fl{\frac{t}{\e^2}})}\int_{B_{R_2}}u_0(y)\d y\Big)e^{-\fl{\frac{t}{\e^2}}} \sum_{n=\fl{\frac{t}{\e^2}}}^{2\fl{\frac{t}{\e^2}}}
\fl{\frac{t}{\e^2}}^n\frac{1}{n!}\\
&\ge \Big(A e^{d t-1}\int_{B_{R_2}}u_0(y)\d y\Big)e^{-m}\sum_{n=m}^{2m}\frac{m^n}{n!}\ge A C_Le^{dt-1} \int_{B_{R_2}}u_0(y)\d y
\end{split}
\end{equation*} 
with $m=\fl{\frac{t}{\e^2}}\ge 1$ and $C_L$ the constant found in
Lemma \ref{2n}. Notice that we have used $\tilde{\e}_0 \leq \e_0$ in
order to apply Lemma \ref{smalle2}, and also
$\tilde{\e}_0 \leq \sqrt{t}$ to make sure that the bound we obtain is
not trivial (that is, to avoid $m=0$).

If $\e\ge \tilde \e_0$ then, applying Lemma \ref{Nlargee} to $\tilde \e_0$, for all $x\in B_\eta$ we obtain
\begin{equation*}\label{eq: finallarge}
    \begin{split}
        u(t,x)&\ge e^{(d-\frac{1}{\e^2})t} \sum_{n=1}^\infty \e^{-2n}\int_0^t\dots\int_0^{t_{n-1}}\ird J_\e^{t_1,\dots,t_n}(e^tx-y)u_0(y)\d y \d t_n\dots \d t_1\\
        & \ge e^{(d-\frac{1}{\e^2})t} \sum_{n=1}^\infty \e^{-2n}\int_0^t\dots\int_0^{t_{n-1}}\int_{B_{R_2}} J_\e^{t_1,\dots,t_n}(e^tx-y)u_0(y)\d y \d t_n\dots \d t_1\\
        &\ge e^{(d-\frac{1}{\e^2})t} \sum_{n=N}^\infty
        \Big(\frac{t}{\e^2}\Big)^{n}\frac{1}{n!}B_n
        \int_{B_{R_2}}u_0(y)\d y
        \\
        &\ge \frac{B}{N!} e^{(d-\frac{1}{\tilde \e_0^2})t} \Big(\frac{t}{\e_0^2}\Big)^{N}\frac{1}{N^{\frac{d}{2}}} \int_{B_{R_2}}u_0(y)\d y.
    \end{split}
\end{equation*}
We finally take
$$\alpha:=\min\Big\{A C_Le^{dt-1}, \frac{B}{N!}
e^{(d-\frac{1}{\e_0^2})t} \Big(\frac{t}{\tilde
  \e_0^2}\Big)^{N}\frac{1}{N^{\frac{d}{2}}}\Big\}.$$
\end{proof}

\begin{remark}
  The previous proof of positivity works almost without changes for
  the equation
  \begin{equation*}
    \partial_t u=\frac{1}{\e^2}(J_\e*u-u)+\div (\nabla V u)
  \end{equation*}
  for any, even \textit{time-dependent}, potential
  $V:\R^d\times\R \to \R$ such that
  $$\frac{d}{\d t}X(t)=\nabla V (X(t),t)$$ is a linear ODE. It seems
  harder to extend  straightforwardly the proof to potentials with different
  expressions.
\end{remark}

\subsection*{Proof of the main result}
We have now all the ingredients to prove Theorem \ref{spectral gap}.

\begin{proof}
  The proof comes straightforwardly thanks to Harris's Theorem
  \ref{sHarris}.  From Theorem \ref{posit}, at every time $t>0$ and
  for every $R_1,R_2>0$, the operator $S_t$ satisfies Hypothesis
  \eqref{eq:HLB}:
  $$S_t u_0 \ge \alpha \mathds{1}_{B_{R_1}}\int_{\mathcal{S}}u_0(y)\d y$$
  on the ``small set'' $\mathcal{S}=B_{R_2}$.

  Theorem \ref{thm: finalLyap} gives the confining Lyapunov conditions in both case $(i)$ and $(ii)$, concluding the proof
\end{proof}

\begin{remark}[On the dependence of the constants]
As mentioned in the introduction, Harris's theorem provides a constructive method that yields explicit—though generally nonoptimal—constants. 
Unsurprisingly, the kernel $J$ plays a central role in determining these constants. Indeed, the latter get better the closer $J$ resembles a Gaussian distribution. This is reflected in our use of the Berry–Esseen theorem, which introduces an explicit dependence on both the Berry–Esseen constant and the number $N$ of convolution iterations. We also note that if $J$ were itself a Gaussian, positivity would follow directly and there would be no need to invoke any central limit theorem. \end{remark}

\section{Convergence: from nonlocal to local}
\label{sec:nonlocal-to-local}

As a consequence of Theorem \ref{spectral gap}, proved in Section
\ref{sec:asymptotic-behavior}, will now prove results on convergence
of the nonlocal Fokker-Planck equation \eqref{NLFP} to the local one
\eqref{eq:FP}. We present proofs which give a speed of convergence of
order $\e^2$, assuming the ``symmetry'' condition \eqref{eq:Jsym} on
the kernel $J$. As we mentioned in Remark
\ref{rem:order-of-convergence}, one may make minor modifications in
the proofs below in order to avoid condition \eqref{eq:Jsym}, at the
price of obtaining a slower speed of convergence of order $\e$. Since
the proofs are very similar, for the sake of brevity we present only
the $\e^2$ case. We will thus assume \eqref{eq:Jsym} for the whole section.

We follow a strategy of proof which is common in numerical
analysis. We  begin by proving suitable consistency and stability
results in order to prove convergence for finite times, assuming nice
regularity and decay hypotheses.

\begin{proposition}[Consistency]\label{prop: consistency}   
Take $k\ge 0$. Assume $J$ satisfies
  \eqref{eq:J-moments-basic}, \eqref{eq:Jsym}, and suppose additionally
  $ J\in L^1_{k+d+4}$.  Let $v\in \mathcal{C}^4\cap L^1_k$ be a function such
  that there exist
  $c > 0$ and $K>k+d$ with $|D^\alpha v(x)| \le c\ap{x}^{-K}$ for
  all multiindices $\alpha$ with $|\alpha|=4$.
  
  Then, there exists a constant $C=C(K,d,k, \norm{J}_{L^1_{4+k+d}},c)>0$ such that
  \begin{equation*}
    \norm{(L_\e-L_0)v}_{L^1_k}\le C\e^2
  \end{equation*}
\end{proposition}
\begin{proof}
    To lighten the notation we denote by $\f=\ap{\cdot}^{-K}$.
We write (understanding that all operators are evaluated at $x \in \R^d$) 
$$
L_\e v -L_0v
=
\frac{1}{\e^2}[J_\e*v-v]-\Delta v=\ird J(z)[\frac{1}{\e^2}(v(x-\e z)-v(x))-\Delta_x v]\d z,$$
and we Taylor-expand $v$ around $x$ up to the fourth order:
\begin{equation*}
  v(x-\e z)=\sum_{|\alpha|\le 3} D^\alpha v(x) \frac{(-\e z)^\alpha}{\alpha!}+\sum_{|\alpha|=4}D^\alpha v(\xi_{x, \e z}) \frac{(-\e z)^\alpha}{\alpha!}
\end{equation*}
with $\xi_{x, \e z}\in[x, x-\e z]$. From assumptions \eqref{eq:J-moments-basic} and \eqref{eq:Jsym}  when integrating in $z$,  for 
$\alpha$ with odd entry $\alpha_i$, all terms in the first sum disappear. This leads to 
\begin{equation*}
\begin{split}
    &\abs{L_\e v-L_0v}=\abs{\ird J(z)[\frac{1}{\e^2}(v(x-\e z)-v(x))-\Delta_x v]\d z}\\
&\le \e^2\ird J(z)\sum_{|\alpha|=4}\frac{1}{\alpha!}\abs{D^\alpha v}(\xi_{x, \e z})|z|^4 \d z  \le c\e^2\int J(z)\f(\xi_{x, \e z})|z|^4\d z.
    \end{split}
\end{equation*}
Thus, we are left with 
\begin{equation*}
\begin{split}
    \norm{(L_\e-L_0)v}_{L^1_k}&\le c\e^2\ird(1+|x|^k)\ird J(z)\f(\xi_{x, \e z})|z|^4 \d z \d x\\
    &= c\e^2\ird J(z)|z|^4\ird \f(\xi_{x, \e z}) (1+|x|^k) \d x \d z.
    \end{split}
\end{equation*}
We split the inner integral:
\begin{equation*}
  \ird \f(\xi_{x, \e z}) (1+|x|^k) \d x
  =
  \int_{|x|\le  2\e |z|}\f(\xi_{x, \e z}) (1+|x|^k) \d x
  +\int_{|x|> 2 \e |z|}\f(\xi_{x, \e z}) (1+|x|^k) \d x.
\end{equation*}
The first term can be bounded by
\begin{equation*}\label{eq: splitint1}
    \int_{|x|\le  2 \e |z|}\f(\xi_{x, \e z}) (1+|x|^k) \d x \le \f(0) (1+|2\e|^k |z|^{k})\abs{B_{2\e |z|}}\le  C_1\e^d|z|^{d}+C_2\e^{k+d}|z|^{ k+d}
\end{equation*}
Since in the second integral $\e |z|<\frac{|x|}{2}$ and $\f$ is radial, we can bound
\begin{equation*}\label{eq: splitint2}
\begin{split}
   &\int_{|x|> 2\e |z|}\f(\xi_{x, \e z}) (1+|x|^k) \d x\le   \int_{|x|> 2\e |z|}\f(|x|-\e|z|) (1+|x|^k) \d x\\
   &\le \int_{|x|> 2\e |z|}\f(x/2)(1+|x|^k)\d x
   \le 2^{k+d}\norm{\f}_{L^1_k}
   \end{split}
\end{equation*}
Then 
\begin{equation*}
    \begin{split}
         &\norm{(L_\e-L_0)v}_{L^1_k}\le c \e^2\ird J(z)|z|^4\big[ 2^{k+d}\norm{\f}_{L^1_k}+ C_2\e^d|z|^{d}+C_3\e^{k+d}|z|^{ k+d}   \big]\d z\\
         &\le c\e^2\Big[C_1\ird J(z)|z|^{4}\d z+C_2\e^d\ird J(z)|z|^{4+d}\d z+C_3\e^{k+d}\ird J(z)|z|^{4+k+d}\d z\Big]\\
         &\le  c \e^2\Big(C_1\norm{J}_{L^1_4}+C_2\e^d\norm{J}_{L^1_{4+d}}+C_3\e^{k+d}\norm{J}_{L^1_{4+k+d}}\Big) \le C\e^2
    \end{split}
\end{equation*}
where $C$ only depends on $c, k, K, d,$, and $\norm{J}_{L^1_{4+k+d}}$.
\end{proof}

\begin{remark}Notice that for any $v$ such that $D^\alpha v=0$ with $|\alpha|\le4$, like polynomials of order $3$ or lower, the local and the nonlocal operators agree.
\end{remark}

We now study the stability of our nonlocal PDE under small
perturbations:

\begin{proposition}[Stability]\label{prop: stability}
    Let $u\in \mathcal{C}([0,\infty);L^1_k )$ be the solution of 
\begin{equation*}        
        \partial_tu=L_\e u       
\end{equation*}
Let $h,v\in \mathcal{C}([0,\infty);L^1_k )$ 
satisfy in the mild sense
\begin{equation*}
    \partial_t v=L_\e v+h(t)
\end{equation*}
with the same initial data $u(0)=v(0)=u_0.$ Then,
there exists  constants $\omega>0,M\ge 1$ depending only on $k$ and $d$ such that 
\begin{equation*}
  \norm{u(t)-v(t)}_{L^1_k}
  \le
  M\int_0^te^{\omega (t-s)}\norm{h(s)}_{L^1_k}\d s
  \qquad \text{for all $t \geq 0$.}
\end{equation*}
\end{proposition}

\begin{proof}
$\psi:=v-u$ solves
\begin{equation*}
  \begin{cases}
    \partial_t \psi=L_\e \psi+h(t)\\
    \psi(0)=0
  \end{cases}
\end{equation*}
Using Duhamel's formula we have 
\begin{equation*}\label{eq: stab1}
\begin{split} 
  \norm{\psi(t)}_{L^1_k}
  &=\norm{\int_0^t e^{(t-s)L_\e}h(s) \d s}_{L_1^k}
    \le
    M \int_0^te^{\omega  (t-s)}\norm{h(s)}_{L^1_k}\d s,
\end{split}
\end{equation*} since for a semigroup there exist $\omega >0$, $M\ge 1$, such that
\begin{equation*}
  \norm{e^{t L_\e} f}_{L^1_k}\le  M e^{\omega t}\norm{f}_{L^1_k}.
  \qedhere
\end{equation*}
\end{proof}

\begin{remark}
  One could improve the growth rate in the above result by a more
  careful estimate of $e^{(t-s)L_\e}h(s)$. However, since the point of
  this is to later obtain a local-in-time bound, we give only the
  above rough estimate.
\end{remark}

We leave here a standard result of the decay of the derivatives of the solution of the Fokker-Planck equation.
\begin{lemma}\label{lemma: decay derivative}
Let $u$ be the solution of Fokker-Planck equation with initial data $u_0$, $K \geq 0$ and $\alpha$ a multiindex. If $u_0$ is such that
$$
\abs{D^\alpha u_0(x)}\le c\ap{x}^{-K}
$$
then for every $T^* > 0$ there exists a constant $C_{T^*}$ depending only on $T^*$, $K$, $\alpha$ and the dimension $d$ such that 
$$
\abs{D^\alpha u(x)}\le C_{T^*} \ap{x}^{-K}
\qquad \text{for all $0 \leq t \leq T^*$.}
$$
\begin{proof}[Sketch of proof]
Since $u(t, x) = e^{dt} v(\frac12 (e^{2t}-1), x e^t) $ where $v$ is a solution to the standard heat equation with the same initial data, it is easier to prove the corresponding result for $v$. Since
$$v = \Gamma_t * u_0,$$
where $\Gamma_t$ is the standard heat kernel, passing the derivative to the initial data one easily obtains the result for $v$. Then going back to $u$ one finds a time dependent constant $C_t$ such that 
$$
\abs{D^\alpha u(x)}\le C_{t} \ap{x}^{-K}
$$
Taking the supremum over $[0,T^*]$ the proof is complete.    
\end{proof}
\end{lemma}

We can therefore prove the following:

\begin{proposition}[Local-in-time convergence]
  \label{prop:convfinite}
  Let $J$ satisfy assumptions \eqref{eq:J-moments-basic},
  \eqref{eq:Jsym}, and assume $J\in L^1_{k+d+4}$ 
    $k \ge 0$.  Let $u$ be the solution of the nonlocal Fokker-Planck
  equation (\ref{NLFP}) and $v$ the solution of the local
  Fokker-Planck equation (\ref{eq:FP}) with the same initial data  $u_0\in L^1_k\cap \mathcal{C}^4$.
  Suppose that there exists
  $c > 0$ and $K>k+d$ with $|D^\alpha u_0(x)| \le c\ap{x}^{-K}$ for
  all multiindices $\alpha$ with $|\alpha|=4$.  Then for all $T^* > 0$ there exist $C, \omega  > 0$ such that,
    \begin{equation*}
      \norm{u(t)-v(t)}_{L^1_k}\le Ce^{\omega T^*}\e^2
      \qquad \text{for $t \in [0,T^*]$.}
    \end{equation*}
\end{proposition}
\begin{proof}
 By Lemma \ref{lemma: decay derivative},  the solution $v=v(t,x)$ of \eqref{eq:FP}
  is in $\mathcal{C}^4$, for every $t\in[0,T^*]$ and 
  $|D^\alpha v(t,x)| \le c\ap{x}^{-K}$ uniformly in $[0,T^*]$,  
 due to our assumptions on the derivatives of $u_0$. Hence, $v$ satisfies the
  hypotheses of Proposition \ref{prop: consistency}, uniformly
    for all $t \in [0,T^*]$.
  We rewrite
\begin{equation*}
    L_0 v=L_\e v+(L_0 v-L_\e v):=L_\e v+h(t).
\end{equation*}
    Then 
    \begin{equation*}
        \norm{u(t)-v(t)}_{L^1_k}\le M e^{\omega t}\int_0^t e^{-\omega s}\norm{L_0v-L_\e v}_{L^1_k}\le M \frac{e^{\omega  t}-1}{\omega }C_1\e^2\le Ce^{\omega T^*}\e^2
    \end{equation*}
    where we used our results on stability of Proposition \ref{prop: stability} in
    the first inequality and the consistency with constant $C_1$ of Proposition \ref{prop: consistency} in the second inequality.
\end{proof}

Once we have studied the convergence for finite times, we  use
Theorem \ref{spectral gap} to gather information on the equilibrium
and the behaviour of the solution at large times. As a first
consequence of the previous results, we can estimate the speed
convergence of the equilibrium $F_\e$ of \eqref{NLFP} to the
equilibrium $G$ of \eqref{eq:FP} as $\e \to 0$. Indeed, having a
spectral gap for an operator (semigroup) provides information of the
norm of the resolvent of the latter, assuring that the norm of the
resolvent remains bounded by a constant.
\begin{proposition}[Convergence of the equilibrium]
  \label{thm: hille}
  Let $k>0$ and suppose $J$ satisfies \eqref{eq:J-moments-basic} and \eqref{eq:Jsym}. Additionally assume $J\in L^1_{k+d+4}$. Then
   \begin{equation*}
       \norm{F_\e-G}_{L^1_k}\le C\e^2
   \end{equation*} 
\end{proposition}
\begin{proof}
  Feller-Miyadera-Philips's general case of Hille-Yosida Theorem for
  semigroups---see \cite[Part 2, Thm. 3.8]{EngelNagel2001}---states
  that having an estimate for a strongly continuous semigroup $(S_t)$ in a Banach space
  $\mathcal{Y}$
of the type
\begin{equation}\label{eq: SGHille}
\norm{S_t}\le Me^{-\gamma t}\qquad \text { for }\gamma>0, M>1,t>0
\end{equation}
(where $\norm{\cdot}$ denotes the operator norm) is equivalent
to the property that for every $\lambda>-\gamma$,
$\lambda $ is in the resolvent set $\rho(L_\e)$ of its infinitesimal generator $(L_\e, \mathcal{D}(L_\e))$, and
\begin{equation}\label{eq: HY}
  \norm{R(\lambda,L_\e)^n}
  \le \frac{M}{(\lambda+\gamma)^n}
  \qquad \text {for every } n\in \mathbb{N},
\end{equation}
where $R$ denotes the resolvent operator.
From Theorem \ref{spectral gap}---in particular from Rmk. \ref{rmk: deacymass0}---the condition \eqref{eq: SGHille} is satisfied in the Banach space $\mathcal{Y}:=\{f\in L^1_k : \ird f=0\}$.
Thus, choosing $n=1$ and $\lambda=0>-\gamma $ in \eqref{eq: HY}, it follows that $L_\e$ has an inverse in $\mathcal{Y}$ with
$\norm{L_\e^{-1}}_{\mathcal{Y}}\le M/\gamma $. This implies that,
\begin{equation*}
\begin{split}
    \norm{F_\e-G}_{L^1_k}&=\norm{F_\e-G}_{\mathcal{Y}}=\norm{L_\e^{-1}L_\e(F_\e-G)}_{\mathcal{Y}}\le
    \norm{L_\e^{-1}}_{\mathcal{Y}}\norm{L_\e (F_\e- G)}_{\mathcal{Y}}
    \\
    &\le
    \frac{M}{\gamma}\norm{L_\e (F_\e-G)}_{L^1_k}=\frac{M}{\gamma}\norm{(L_\e-L_0) G}_{L^1_k},
\end{split}
\end{equation*}
since $L_\e F_\e=L_0G=0$.
Given the regularity and decay properties of the Gaussian, by
consistency \ref{prop: consistency} we obtain
    \begin{equation*}
        \norm{(L_\e-L_0)G}_{L^1_k}\le C\e^2,
    \end{equation*}
    concluding the proof.
\end{proof}
As a consequence of the previous proposition and, again, Theorem \ref{spectral gap} we can estimate the speed of convergence at large times.
\begin{proposition}[Convergence for large times]
  \label{prop:convlargeT}
  Let $J$ satisfy the hypotheses of Theorem \ref{spectral gap} $(i)$,
  Assumption \eqref{eq:Jsym} and, additionally $J\in L^1_{k+d+4}$. Let
  $u$ be the solution of \eqref{NLFP} and $v$ be the solution of
  \eqref{eq:FP} with the same initial data $u_0\in L^1_k$. Then, there
  exist $T^{*}>0$ and $C > 0$ (depending on $d$, $k$,
  $p$, and $u_0$) such that
  $$\norm{u(t,\cdot)-v(t,\cdot)}_{L^1_k}\le C\e^2
  \qquad \text{ for all $t>T^*$.}$$
\end{proposition}

\begin{proof}%
From Theorem \ref{spectral gap}
\begin{equation*}
    \norm{u(t,\cdot)-F_\e}_{L^1_k}\le C_1e^{-\lambda_1 t}\norm{u_0-F_\e}_{L^1_k}.
\end{equation*}
Moreover, we have a similar result for the local case--- see e.g. \cite{Gualdani2018}.
\begin{equation*}
    \norm{v(t,\cdot)-G}_{L^1_k}\le C_2e^{-\lambda_2 t}\norm{u_0-G}_{L^1_k}
\end{equation*}
where $v$ is the solution of the local Fokker Planck equation with the
same initial data. This means
$$\norm{u(t,\cdot)-v(t,\cdot)}_{L^1_k}\le \norm{F_\e-G}_{L^1_k}+Ce^{-\lambda t} \Big(\norm{u_0-F_\e}_{L_1^k}-\norm{u_0-G}_{L_1^k}\Big).$$
From Theorem \ref{thm: hille},
$$ \norm{F_\e-G}_{L^1_k}\le \tilde C\e^2.$$
Let us define $T^*$ such that 
$$e^{-\lambda T^*} \Big(\norm{u_0-F_\e}_{L^1_k}-\norm{u_0-G}_{L^1_k}\Big)\le \e^2.$$
Then for all $t>T^*$
\begin{equation*}
  \norm{u(t,\cdot)-v(t,\cdot)}_{L^1_k}\le C\e^2.
  \qedhere
\end{equation*}
\end{proof}
We finally combine the previous results to give the proof of Theorem \ref{thm:asymptconv}:

\begin{proof}[Proof of Theorem \ref{thm:asymptconv}]
  The proof comes directly from the previous estimates in this
  section. Indeed, take $T^*$ as in Proposition \ref{prop:convlargeT}.
  If $t\in[0,T^*]$, then apply Proposition \ref{prop:convfinite}. If
  $t>T^*$, then Proposition \ref{prop:convlargeT} gives the result.
\end{proof}
\section{Analysis of the equilibrium: regularity, moments, cumulants
  and tails}
\label{sec: moments}

Thanks to Harris's Theorem we have proved the existence and uniqueness
of an exponentially stable equilibrium for the equation \eqref{NLFP}
at each $F_\e$. We can infer some information about its regularity and
shape, especially on the cumulants, moments and tails.

\subsection{Regularity}
\label{sec:regularity}

Thanks to the  explicit form \eqref{eq: eqfourier} of the
equilibrium in the frequency domain, one can study the regularity by
standard results on the Fourier transform.

Recall that  the Fourier transform of the equilibrium is 
\begin{equation*}
  \hat{F}_\e(\xi)=\exp\Big(\frac{1}{\e^2}\int_0^\infty\zeta_\e(e^{-s}\xi)\d s\Big)
\end{equation*}
with $\zeta_\e(\xi)=\hat{J}_\e(\xi)-1$.  
\begin{theorem}
    The equilibrium $F_\e$ of equation \eqref{NLFP} has the following regularity properties. For all $\e\in(0,1]$
    \begin{enumerate}[(i)]
        \item $F_\e\in \mathcal{C}^{k-d}$ for all $k<\frac{1}{\e^2}$ such that $k\ge d$.
        \item $F_\e\notin \mathcal{C}^k$ for all $k>\frac{1}{\e^2}$.
    \end{enumerate}
\end{theorem}
\begin{proof}Let us assume $\zeta_\e$
radially symmetric. If not, define
$\underline{\zeta}_\e$ and $\overline\zeta_\e$ by
\begin{equation*}
    \underline{\zeta}_\e(z)=\min_{|y|=|z|}\zeta_\e(y)\qquad \overline{\zeta}_\e(z)=\max_{|y|=|z|}\zeta_\e(y).
\end{equation*}
These functions are then radially symmetric and satisfy
\begin{gather*}
  \underline{\zeta}_\e(\xi)\le \zeta_\e(\xi)\le
  \overline\zeta_\e(\xi),
  \\
  \underline{\zeta}_\e(\xi)=-|\xi|^2+o(|\xi|^3)\quad \text{and} \quad
  \overline{\zeta}_\e(\xi)=-|\xi|^2+o(|\xi|^3)
  \qquad \text{as $\xi \to 0$.}
\end{gather*}
After a change of
variable $y=e^{-s}|\xi|$, we obtain
\begin{equation*}\label{eq: zeta}
  \int_0^\infty \zeta_\e(e^{-s}|\xi|)\d s=\int_0^{|\xi|}\frac{\zeta_\e (y)}{y}\d y
\end{equation*}
Since $\zeta_\e(y)=-\e^2y^2+o(y^2)$, the singularity at zero is
canceled and, moreover, close to $\xi=0$ we have
$$\hat F_\e(\xi)\approx e^{-\frac{|\xi|^2}{2}}.$$
To study the regularity, we are interested in the decay properties of
the Fourier transform
\begin{enumerate}[(i)]
    \item To prove the first claim, it suffices to show that \begin{equation}\label{eq: condreg1}
    |\hat F_\e(\xi)|\le \frac{c}{1+|\xi|^{k+\eta}}\qquad\text{ for }k<\frac{1}{\e^2}\end{equation}
    for some (small) positive $\eta$.
    Indeed, by the Riemann-Lebesgue lemma, $\hat{J}(y)\to 0$ as
    $|y|\to \infty$. Let $y_\delta > 0$ be such that 
    \begin{equation*}
        |\hat{J}(y)|\le \delta:=\frac{\e^{2}}{2}\Big(\frac{1}{\e^2}-k\Big)\qquad\text{
          for all }|y|>y_\delta
    \end{equation*}
    Without loss of generality, we can assume $y_\delta<|\xi|$,
    otherwise there is nothing to prove.  Then
    \begin{equation*}
    \begin{split}
        \int_0^{|\xi|}\frac{\zeta_\e (y)}{y}\d y= \int_0^{y_\delta}\frac{\zeta_\e (y)}{y}\d y+\int_{y_\delta}^{|\xi|}\frac{\zeta_\e (y)}{y}\d y\le C_1-(1-\delta)(\log |\xi|-\log|y_\delta|)
        \end{split}
    \end{equation*}
    which implies that 
    $$|\hat{F}_\e(\xi)|\le C\exp\Big(-\frac{1-\delta}{\e^2}\log|\xi|\Big)=\frac{C}{|\xi|^{\frac{1-\delta}{\e^2}}}\le\frac{C}{|\xi|^{k+(\frac{1}{\e^2}-k)/2}}.$$
    For $|\xi|$ large, the previous implies condition \eqref{eq: condreg1}
    with $\eta< (\frac{1}{\e^2}-k)/2$ and by standard argument on Fourier  transform $F_\e\in \mathcal{C}^{k-d}$ for all $k<\frac{1}{\e^2}$.
  \item To prove the second claim, we will show that 
    \begin{equation*}
      |\xi|^k|\hat{F}_\e(\xi)|\notin L^\infty\qquad\text{ for }k>\frac{1}{\e^2}
    \end{equation*}
    proceeding in a similar way as before. Notice that since
    $\xi \mapsto \hat{J}(\xi)$ is continuous and decays to zero as
    $|\xi| \to +\infty$, there exists $y_\delta$ such that we can
    bound $\hat{J}$ from below by
    $$\hat{J}(x)\ge- \delta :=\frac{\e^{2}}{2}\Big(\frac{1}{\e^2}-k\Big)\qquad\text{ for
      all } y\ge y_\delta.$$ As a consequence we see that
    \begin{equation*}
      \begin{split}
        |\xi|^k\hat{F}_\e(\xi)\ge |\xi|^k  C\exp\Big(-\frac{1+\delta}{\e^2}\log|x|\Big)=|\xi|^{k-\frac{1+\delta}{\e^2}}\notin L^\infty
      \end{split}
    \end{equation*}
    and thus by standard properties of the Fourier transform,
    $F_\e\notin \mathcal{C}^k$.
\end{enumerate}
\end{proof}
    We just proved that the regularity keeps improving when $\e\to 0$ although the result is not optimal. A precise characterization remains, to our knowledge, an open question.
Moreover, we suspect that under quite general
conditions, the only point at which $F_\e$ is not $\mathcal{C}^\infty$ is the
origin. We are able to show this in dimension $1$, and for radially
symmetric $J$:
  \begin{theorem}
    Let $\Omega\subset\subset \R^d$ such that $0\notin\Omega$. Suppose (at least) one of the two following condition is verified
    \begin{enumerate}[(a)]
        \item $d=1$
        \item $J$ is radially symmetric.       
    \end{enumerate}
     then $F_\e\in \mathcal{C}^\infty(\R^d\setminus \{0\}).$
\end{theorem}
\begin{proof}
    We proceed via a bootstrap argument: the equilibrium satisfies the stationary equation
    \begin{equation}\label{eq: stationary}
        \div(xF_\e)=-\frac{1}{\e^2}(J_\e*F_\e-F_\e)
    \end{equation}
    Since $F_\e\in L^1(\Omega)$, the RHS is also in $L^1(\Omega)$, and
    thus $\div(x F_\e)$ must be in $L^1(\Omega)$ as well.
\begin{enumerate}[(a)]
    \item If $d=1$ then $\div(xF_\e)\in L^1(\Omega)$ and $F_\e\in L^1(\Omega)$ implies that $xF'\in L^1(\Omega)$ and, being $\dist(0,\Omega)\ge \delta$, also $F'$ is integrable. This, by the fundamental theorem of calculus, leads to continuity for $F$ and consequently for the RHS of \eqref{eq: stationary}. Iterating the process one gets  the result.
    \item If $d\ne 1$, the fact that $\div(xF_\e)$ is integrable in $\Omega$ does not assure that $\nabla F\in L^1(\Omega)$, even when we are far from $0$. Therefore, we need an additional hypothesis.
    If $J$ is radially symmetric,  $F_\e$ inherits this property (see Remark \ref{rmk: propinh}), and for such functions the following holds:
    $$
    F_\e\in W^{j,1}(\Omega),\quad \div(xF_\e)\in W^{j,1}(\Omega) \implies F_\e\in W^{j+1,1}(\Omega)\qquad\text{ for all }j\ge 0
    $$
    Indeed, since
    $\div(xF_\e)=dF_\e+x\cdot \nabla F_\e$ and, by hypotheses $x\cdot F_\e\in W^{1,1}(\Omega)$, due to radial symmetry we can write
    $$x\cdot \nabla F_\e= x\cdot F_\e'(|x|)\frac{x}{|x|}=F_\e'(|x|)|x|.$$
    
    Since $0\notin \Omega$, (i.e.$|x|\ge \delta $), it follows that $F'(|x|)\in W^{1,1}(\Omega)$. By definition of $\nabla F_\e$ and again the fact that $0\notin \Omega$, we conclude that  $\nabla F_\e\in W^{j,1}(\Omega)$, which implies that $ F_\e\in W^{j+1,1}(\Omega)$.

    We then iterate the process as in step $(a)$ using \eqref{eq:
      stationary}, improving the regularity up to
    $F_\e\in W^{d,1}(\Omega)$, where $d$ is the dimension. By the
    Sobolev embedding theorem, this implies that $F_\e\in \mathcal{C}(\Omega)$.
 
    If we keep repeating the argument $(k+1)d$ times, we have that
    $F_\e\in \mathcal{C}^k(\Omega)$ for every $k$ and thus the $\mathcal{C}^\infty$
    regularity follows.
\end{enumerate}
\end{proof}

\subsection{Moments and cumulants}

Let $\alpha=(\alpha_1,\dots,\alpha_d)$ be a multi-index with order
$|\alpha|=n$. For a nonnegative Borel measure $f$ on $\R^d$ We define
the \emph{absolute moment of order $n$} by 
$$M_n(f) := \ird |x|^{n}f(x)\d x,$$
(always well-defined but possibly $+\infty$) and the $\alpha$-moment
by
$$m_\alpha(f) := \ird x^\alpha f(x)\d x  =\ird x_1^{\alpha_1}\cdots
x_d^{\alpha_d}f(x) \d x,$$ which is well-defined whenever
$M_n[f] < +\infty$. The \emph{moment generating function}
$\mathfrak{M}(\xi)$ is given by
\begin{equation*}
  \mathfrak{M}[f](\xi)
  :=
  \ird e^{x \xi} f(t,x)  \d x
\end{equation*}
whenever $x \mapsto e^{x \xi} u(t,x)$ is integrable. Moments can then
also be defined via the moment generating function by
$$m_\alpha(f)=D^\alpha\mathfrak{M}[f](\xi) \Big\vert_{\xi=0}.$$

\emph{Cumulants} are an alternative way to describe a distribution,
similar to the moments. They can be defined through the
\emph{cumulative generating function (c.g.f.)} $C[f]$, which is the logarithm
of the moment generating function $\mathfrak{M}(\xi)$:
\begin{equation*}
C[f](\xi) := \log \ird e^{x \xi} f(x)  \d x.
\end{equation*}
The $\alpha-$cross cumulant is then defined by
$$\kappa_\alpha(f)=D^\alpha C[f](\xi) \Big\vert_{\xi=0}.$$
If two distributions have the same moments then they have the same
cumulants and vice versa. However, moments always depend on the
moments of lower order while the cumulants don't.

Using  F\'aa di Bruno's formula, a generalization of
the chain rule for repetitive derivatives, one can prove that cumulants and moments are linked
 via Bell's polynomials,
defined in \cite{comtet_advanced_1974}. These are well known combinatorial
objects in $\R$, defined through their exponential generating
function, that is
\begin{equation}
  \label{eq: bells}
  \exp\left(\sum_{i=1}^\infty \frac{x_i }{ i!} y^i \right)
  = \sum_{n=0}^\infty \frac{B_n(x_1,\dots,x_n)}{n!} y^n.
\end{equation}

In the rest of this section, $u$ denotes a solution to (\ref{NLFP}),
and $C = C(t,\xi)$ is its cumulative generating function:
\begin{equation*}
  C(t,\xi) = \log \ird e^{x \xi} u(t,x)  \d x.
\end{equation*}
This can be explicitly computed in a simple way:
\begin{theorem}
  \label{cgf} Let $J$ satisfy \eqref{eq:J-moments-basic} and let
 $u$ be the solution of \eqref{NLFP}
  with initial data $u_0$. Assume that  $\ird e^{|x|}u_0\d x<\infty$ and
  $\ird e^{|x|}J\d x<\infty$.  Then
$$C(t,\xi)=C_0(e^{-t}\xi)+\frac{1}{\e^2}\int_0^t\zeta(e^{-s}\xi)\d s $$
with $C_0(\xi)=\log\ird e^{x\xi}u_0(x)\d x$ and
$\zeta(\xi)=\mathfrak{M}(J)-1=\ird e^{x\xi}J(x)\d x-1$.
\end{theorem}

\begin{proof}
  We proceed similarly to the work of \cite{alfaro_density_2020} and
  \cite{gil_mathematical_2017}. The partial derivatives of the
  c.g.f. are
$$\partial_{\xi_i}C(t,\xi)=\frac{\ird x_ie^{x\cdot \xi}\d x}{\ird e^{x\cdot \xi}\d x}$$
and
$$ \partial_t C(t,\xi)=\frac{\ird e^{x\xi} \partial_tu \d x}{\ird e^{x\xi} u \d x}=\frac{1}{\e^2} \frac{\ird e^{x\xi} (J_\e*u-u)\d x}{\ird e^{x\xi} u \d x}+\frac{\ird e^{x\xi} \div(xu)\d x}{\ird e^{x\xi} u \d x}$$
Analyzing the first term we get
		\begin{equation*}
			\begin{split}
				&\ird e^{x\xi}(J_\e*u)\d x=\ird e^{x \xi}\Big(\ird J_\e(x-y) u(t,y)\d y\Big)\d x
				\\
				&=\ird e^{y\xi}u(t,y)\d y\ird J_\e(x-y)e^{(x-y)\xi}\d x =\Big(\ird e^{y \xi}u(t,y)\d y\Big)\mathfrak{M}(J_\e).
			\end{split}    
		\end{equation*}
                For the second term, integrating by parts,
$$
	\ird e^{x \cdot \xi }\div_x(xu)\d x=-\ird \nabla_x(e^{x\cdot \xi}) \cdot(x u)=-\xi\cdot \ird e^{x \cdot \xi}x u=-\xi\cdot \nabla_\xi C(t,\xi)\ird e^{x\xi} u \d x
	$$

Therefore\begin{equation*}
    \partial_t C(t,\xi)=\frac{1}{\e^2}\zeta_\e(\xi)-\xi\cdot \nabla_\xi C(t,\xi).
\end{equation*}
One can check that the only solution to this above first order PDE is
the function $C$ in the statement (of course, the PDE can be solved,
for example by the method of characteristics, to obtain the given
formula). Hence the proof is complete.
\end{proof}

\begin{corollary}
 \label{prop:mom_bell_Rd}
 In the hypotheses of the previous theorem,
let $\alpha$ be a multi-index with $|\alpha|=n$. Then,
\begin{enumerate}[(i)]
    \item
   The $\alpha$-cross cumulant is
    \begin{equation*}
    \kappa_\alpha (u(t))=e^{-|\alpha|t}\kappa_\alpha (u_0)+ (1-e^{-|\alpha|t})\frac{1}{|\alpha|\e^2}m_\alpha(J_\e)
    \end{equation*}
 As a consequence 
    \begin{equation*}
     \kappa_\alpha(F_\e) = 
        \frac{1}{|\alpha|\e^2}m_\alpha(J_\e)
\end{equation*}
\item If $d=1$, then, the $n-th$ moment is
    \begin{equation}\label{cummomRd}
    m_n(u)=B_n(\kappa_1,\dots,\kappa_n)=\sum_{k=1}^n B_{n,k}(\kappa_1,\dots,\kappa_{n-k+1})
\end{equation}
where $B_n$ are the complete (exponential)
Bell's Polynomials defined in \eqref{eq: bells} and $B_{n,k}$ are called partial Bell's Polynomials\footnote{From a combinatorial point of view, $B_{n,k}(x_1,\dots, x_{n-k+1}$) enumerates the possible ways to partition a set of 
$
n$ elements into 
$k$ non-empty subsets. Each term of the polynomials represents  specific way of distributing the $n$ elements in the $k$ subsets. For example, $B_{5,2}=10x_2x_3+5x_1x_4$ means that there are $10$ ways to partition a set of $5$ elements into two subsets of cardinality respectively $2$ and $3$ and $5$ ways to do so in two subsets of cardinality $1$ and $4$.
}.

 \end{enumerate}

\end{corollary}
\begin{proof}
    From Theorem \ref{cgf} the $\alpha$-th derivative of the c.g.f.~is
    \begin{equation*}
      D^{\alpha} C(t,\xi)=e^{-nt}(D^\alpha C_0) (e^{-t}\xi)+\frac{1-e^{-nt}}{n\e^2}(D^\alpha \beta) (e^{-t}\xi)
    \end{equation*}
    and thus, evaluating at $\xi=0$ we get the $\alpha$-mixed cumulant
    \begin{equation*}
        \kappa_\alpha (u(t))=e^{-nt}\kappa_\alpha (u_0)+ (1-e^{-nt})\frac{1}{n\e^2}m_\alpha(J_\e)
    \end{equation*}
    Expression \eqref{cummomRd} comes directly from the definition of
    Bell's polynomials \eqref{eq: bells}. 
  \end{proof}

\begin{remark}
       One can generalize the last result to higher dimension but it appears that the multidimensional version of Bell's polynomials is not very well known. We  refer to   \cite{Withers2010}, \cite{schumann2019multivariate} and references therein for further information.
\end{remark}
\begin{remark}
 By computing explicitly the time derivative of the
    $\alpha$ moment, i.e.
    \begin{equation*}\label{34b}
      \begin{split}
        \ddt m_{\alpha}(u(t))
        =\frac{1}{\e^2}(m_{\alpha}(J_\e*u)-m_{\alpha}(u))-m_{\alpha}(\div
                            (xu)),
      \end{split}
    \end{equation*}
     one can obtain the following recursive formula for the moments of the solution (and consequently for the equilibrium):
 \begin{equation*}
     m_{\alpha}(u(t))=
         \gamma_\alpha(1-e^{-nt})+e^{-nt}m_{\alpha}(u_0) 
    \end{equation*}
    where
    \begin{equation*}
        \gamma_\alpha=\frac{1}{n\e^2}\sum_{\substack{\beta\le\alpha\\\beta\ne 0}}^{n}\binom{\alpha}{\beta}m_{\beta}(J)m_{\alpha-\beta}(u).
    \end{equation*}

\end{remark}

\subsection{Tails}
\label{sec:tails}

For the usual (local) equilibrium for the Fokker-Planck equation
\eqref{eq:FP}, is a standard Gaussian. In the nonlocal approximation
in this paper, the tails of the equilibrium seem to decay slower, no
matter how fast the kernel $J$ is decaying. As a consequence
of Harris's theorem and our Lyapunov condition estimates (Theorem \ref{thm:asymptconv} \textit{(ii)}), we know for an
exponentially decaying kernel $J\in L_{\exp,a'}$, the unique
probability equilibrium has exponentially decaying tails; that is, for
every $a<a'\e^{-1}$
$$\ird F_\e(x)e^{a\ap{x}}\d x<\infty.$$
We conjecture that for a kernel $J$ decaying fast enough, the tails of
the equilibrium have a ``Poisson-like'' behaviour: they decay roughly
like $e^{-|x|\log (|x|+1)}$. This implies that
$$\ird F_\e(x)e^{a \ap{x}^\gamma}\d x=\infty\quad \text{ for every }\gamma>1.$$
We have not been able to prove this ``Poisson-like'' behavior in
general, but we can give a good upper bound when $J$ has compact
support:

\begin{proposition}
     Let $J:\R^d\to \R$ be compactly supported and its support is in the ball $B_R$ for some $R>0.$ Then 
    \begin{equation*}
        \int F_\e(x)e^{a\ap{x}\log\ap{x}}\d x<\infty
    \end{equation*}
    for any $a\le\frac{1}{\e R}$.
\end{proposition}

\begin{proof}
  Let $\f(x)=\exp( a\ap{x}\log\ap{x} )$ with $a$ to be determined
  later. In the same spirit of the proofs of theorems \ref{thm:lyapunovpow} and \ref{thm:lyapunovexp}, one can compute \begin{equation*}\begin{split}
            &\frac{1}{\e^2}\ird J_\e(y)(\f(x+y)-\f(x))\d y=\frac{1}{\e^2}\ird J(z)(\f(x+\e z)-\f(x))\d z\\
            &=\frac{1}{\e^2}\ird J(z)\Big(\f (x)+\frac{1}{2}\sum_{|\alpha|=2}R_\alpha(\xi)\e^2 z^{\alpha}-\f(x)\Big)\d z=\frac{1}{2}\ird J(z) \sum_{i=1}^d \partial_i^2\f(\xi)z_i^{2}\d z.
          \end{split}
	\end{equation*}
        With a simple calculation
        \begin{equation*}
          \begin{split}
            \sum_{i=1}^d\partial_i^2\f(x)&=\sum_{i=1}^da\frac{\f}{\ap{x}^3}\Big(a^2x_i^2\ap{x}(\log^2\ap{x}+2\log\ap{x}+1)+(1+\sum_{j\ne 1}x_j^2)(\log\ap{x}+1)\Big)\\
            &\le C_1a^2\log^2\ap{x}\f(x),
          \end{split}
 \end{equation*}
 and since $\f$ is radially increasing
 $\ap{|x|+|z|}\le \ap{x}+\ap{z}$, and
 $\log^2\ap{x+z}\lesssim \log^2{\ap{x}}+\log^2\ap{z}$ we have
 \begin{equation*}
     \begin{split}
        \frac{1}{\e^2}&\ird J_\e(y)(\f(x+y)-\f(x))\d y
         \le\frac{C_1}{2}a^2\int_{B_R}J(z) \log^2\ap{\xi}\f(\xi) \d z\\
         &\le  C_2a^2  \log^2\ap{x}\f(|x|+\e R)\int_{B_R} J(z) \d z= C_2a^2  \log^2\ap{x}\ap{|x|+\e R}^{a\ap{|x|+\e R}}\\
    &\le C_2a^2  \log^2\ap{x}\ap{|x|+\e R}^{a\e R}e^{a\ap{x}\log\ap{x}}=\Big[ C_2a^2  \log^2\ap{x}\ap{|x|+\e R}^{a\e R}\Big]\f(x)
     \end{split}
 \end{equation*}
  Moreover,
\begin{equation*}
\begin{split}
    x\cdot \nabla\f&=x\cdot a\frac{x}{\ap{x}}(\log\ap{x}+1)\f(x)\le\Big[ C_3a\ap{x}\log\ap{x}\Big]\f(x)
    \end{split}
\end{equation*}
Consequently,
	\begin{equation}\label{eq: LyappowfinalPoiss}
		\begin{split}
			L^*_\e&\f(x)=\frac{1}{\e^2}\ird J_\e(y)(\f(x+y)-\f(x))\d y-x\cdot \nabla \f (x)\\
            &\le  \Big[C_2a^2  \log^2 \ap{x}\ap{|x|+\e R}^{a\e R}- C_3 a\ap{x}\log\ap{x}\Big]\f(x)\le    C-\lambda \f
		\end{split}
	\end{equation}
        for some $C$ and $\lambda>0$ and $a < \frac{1}{\e R}$.
\end{proof}

\section*{Acknowledgments}
We would like to thank Josephine Evans for discussions which gave rise
to the early ideas of this paper. We would also like to thank
Alejandro Gárriz for several discussions and suggestions.

The project that gave rise to these results received the support of a
fellowship from the ``la Caixa'' Foundation (ID 100010434). The
fellowship code is LCF/BQ/DI22/11940032. Both authors acknowledge
support from the ``Maria de Maeztu'' Excellence Unit IMAG, reference
CEX2020-001105-M, funded by MCIN/AEI/10.13039/501100011033/. They were
also supported by grant PID2023-151625NB-I00 and the research network
RED2022-134784-T from MCIN/AEI/10.13039/501100011033/.

\addcontentsline{toc}{section}{References}
\bibliography{bibliography}
\end{document}